\documentclass[12pt,reqno]{amsart}
\usepackage[colorlinks=true, pdfstartview=FitV, linkcolor=blue, citecolor=blue, urlcolor=blue]{hyperref}

\usepackage{amssymb,amsmath, amscd}
\usepackage{bbm}
\usepackage{times, verbatim}
\usepackage{graphicx}
\usepackage[english]{babel}
 \usepackage[usenames, dvipsnames]{color}
\usepackage{amsmath,amssymb,amsfonts}
\usepackage{enumerate}
\usepackage{anysize}
\marginsize{4cm}{4cm}{4cm}{4cm}
\input xy
\xyoption{all}
\usepackage{pb-diagram}
\usepackage[all]{xy}
\input xy
\xyoption{all}

\DeclareFontFamily{OT1}{rsfs}{}
\DeclareFontShape{OT1}{rsfs}{n}{it}{<-> rsfs10}{}
\DeclareMathAlphabet{\mathscr}{OT1}{rsfs}{n}{it}

\begin{document}

\theoremstyle{plain}

\newtheorem{theorem}{Theorem}[section]
\newtheorem{thm}[equation]{Theorem}
\newtheorem{proposition}[equation]{Proposition}
\newtheorem{corollary}[equation]{Corollary}
\newtheorem{conj}[equation]{Conjecture}
\newtheorem{lemma}[equation]{Lemma}
\newtheorem{definition}[equation]{Definition}
\newtheorem{question}[equation]{Question}

\theoremstyle{definition}
\newtheorem{conjecture}[theorem]{Conjecture}

\newtheorem{example}[equation]{Example}
\numberwithin{equation}{section}

\newtheorem{remark}[equation]{Remark}

\newcommand{\Hom}{{\rm Hom}}
\newcommand{\Aut}{{\rm Aut}}
\newcommand{\rk}{{\rm rk}}
\newcommand{\Ext}{{\rm Ext}}
\newcommand{\Ind}{{\rm Ind}}
\newcommand{\ind}{{\rm ind}}
\newcommand{\Irr}{{\rm Irr}}

\def\G{{\rm G}}
\def\St{{\rm St}}
\def\GL{{\rm GL}}
\def\SO{{\rm SO}}

\def\Ext{{\rm Ext}}
\def\Hom{{\rm Hom}}
\def\Alg{{\rm Alg}}
\def\GL{{\rm GL}}
\def\SO{{\rm SO}}
\def\G{{\rm G}}
\def\U{{\rm U}}
\def\St{{\rm St}}
\def\Wh{{\rm Wh}}
\def\RS{{\rm RS}}
\def\ind{{\rm ind}}
\def\Ind{{\rm Ind}}
\def\csupp{{\rm csupp}}

\title{On higher multiplicity upon restriction \\ from $\GL(n)$ to $\GL(n-1)$}
\author{Mohammed Saad Qadri}
\address{Department of Mathematics \\
Indian Institute of Technology Bombay \\ Mumbai 400076}
\email{185090012@iitb.ac.in}
\date{}

\maketitle

\begin{abstract}
Let $F$ be a non-archimedean local field. Let $\Pi$ be a principal series representation of $\GL_n(F)$ induced from an irreducible cuspidal representation of a Levi subgroup. When $\pi$ is an essentially square integrable representation of $\GL_{n-1}(F)$ we prove that $\Hom_{\GL_{n-1}}(\Pi,\pi)$  $= \mathbb{C}$ and $\Ext^i_{\GL_{n-1}}(\Pi,\pi) = 0$ for all integers $i\geq 1$, with exactly one exception (up to twists), namely, when $\Pi= \nu^{-(\frac{n-1}{2})} \times 
\nu^{-(\frac{n-3}{2})} \times \ldots \times \nu^{(\frac{n-1}{2})}$ and $\pi$ is the Steinberg representation. When $\Pi= \nu^{-(\frac{n-1}{2})} \times 
\nu^{-(\frac{n-3}{2})} \times \ldots \times \nu^{(\frac{n-1}{2})}$ and $\pi$ is the Steinberg representation of $\GL_{n-1}(F)$, then $\dim \Hom_{\GL_{n-1}(F)}(\Pi,\pi)=n$. We also exhibit specific principal series for which each of the intermediate multiplicities $2, 3, \ldots, (n-1)$ are attained.

Along the way, we give a complete list of irreducible non-generic representations of $\GL_{n}(F)$ that have the Steinberg representation of $\GL_{n-1}(F)$ as a quotient upon restriction to $\GL_{n-1}(F)$. We also show that there do not exist non-generic irreducible representations of $\GL_{n}(F)$ that have the generalized Steinberg representation as a quotient upon restriction to $\GL_{n-1}(F)$.
\end{abstract}

\section{Introduction}

Let $F$ be a non-archimedean local field. Let $\Pi$ be a smooth finite-length representation of $\GL_{n}(F)$. We consider $\GL_{n-1}(F)$ as a subgroup of $\GL_{n}(F)$ via the natural embedding $g \rightarrow \begin{pmatrix}
g & 0 \\
0 & 1\\
\end{pmatrix}$. Branching laws about the restriction of $\Pi$ from $\GL_{n}(F)$ to $\GL_{n-1}(F)$ have been a topic of much study. 

One of the basic results in the subject is the multiplicity one theorem for irreducible representations due to Aizenbud-Gourevitch-Rallis-Schiffman \cite{agrs10}, according to which if $\Pi$ is an irreducible representation of $\GL_{n}(F)$ and $\pi$ is an irreducible representation of $\GL_{n-1}(F)$ then, \[ \dim \Hom_{\GL_{n-1}(F)}(\Pi,\pi) \leq 1. \]

More recently, this multiplicity one theorem has been extended by K.Y. Chan to the case when $\Pi$ is a standard representation of $\GL_{n}(F)$ and $\pi$ is the dual of a standard representation of $\GL_{n-1}(F)$ (see \cite[Theorem 1.1]{ch23}) in which case it was shown that,  \[ \dim \Hom_{\GL_{n-1}(F)}(\Pi,\pi) = 1, \] and for all integers $i\geq 1$, \[ \Ext^i_{\GL_{n-1}(F)}(\Pi,\pi) =0. \]

For the above result, it is essential that $\Pi$ and the dual of $\pi$ are standard modules and this ``ordering" of the factors inside the principal series plays a crucial role in the proof of the above result. Since any irreducible representation of $\GL_{n-1}(F)$ is the quotient of a standard representation, and hence any irreducible representation of $\GL_{n-1}(F)$ is a submodule of the dual of a standard module of $\GL_{n-1}(F)$, so one obtains as a corollary of the above result of Chan that if $\Pi$ is a standard representation of $\GL_{n}(F)$ and $\pi$ is an irreducible representation of $\GL_{n-1}(F)$, then \[ \dim \Hom_{\GL_{n-1}(F)}(\Pi,\pi) \leq 1. \] 

A natural question that now arises is the following. Suppose $\Pi$ is a principal series representation of $\GL_{n}(F)$ induced from an irreducible representation of a Levi subgroup. Let $\pi$ be an irreducible representation of $\GL_{n-1}(F)$ as before. Can we determine the dimension of $\Hom_{\GL_{n-1}(F)}(\Pi,\pi)$ or at least give an upper bound on the dimension of this $\Hom$ space? Since any irreducible representation is a quotient of a principal series representation obtained from a cuspidal representation of a Levi subgroup, for the purpose of obtaining an upper bound, it is enough to assume that the principal series representation of $\GL_n(F)$ is induced from an irreducible cuspidal representation of a Levi subgroup inside $\GL_n(F)$.

Such an analysis of multiplicity for a reducible principal series representation has been previously considered in the work of C. G. Venketasubramanian \cite{vs13} where it was shown that if $\pi$ is the trivial representation of $\GL_{n-1}(F)$ then (see \cite[Theorem 7.17]{vs13}), \[ \dim \Hom_{\GL_{n-1}(F)}(\Pi,\pi) \leq 2. \] Moreover, in \cite[Theorem 7.17]{vs13} a complete list of principal series for which higher multiplicity ($= 2$) occurs is also obtained.

In this work, we study the multiplicity question for a general principal series representation $\Pi$ of $\GL_n(F)$ induced from an irreducible cuspidal representation of a Levi subgroup. We restrict our attention to the case when $\pi$ is a generic representation of $\GL_{n-1}(F)$. Since a general principal series representation $\Pi$ of $\GL_{n}(F)$ may be reducible, we no longer expect the multiplicity one property to hold but one may hope to obtain an upper bound on the dimension of the $\Hom$ space.

This paper is written with two goals in mind. The first goal is to identify situations beyond the work of K.Y. Chan involving standard modules where the $\Hom$ multiplicity one property holds. The second goal is to identify situations where higher multiplicity occurs and to obtain an upper bound on the dimension of the $\Hom$ space in such situations. We proceed with the hope that an answer to the former question may allow us to obtain a suitable answer to the latter question. 

The $\Hom$ multiplicity one result of K.Y Chan for standard representations suggests that the ordering of the factors in the principal series $\Pi$ plays an important role in deciding whether higher multiplicity occurs or not. Indeed, our first main result (Theorem \ref{multiplicity one for generic}) confirms this. Given a principal series $\Pi$ of $\GL_{n}(F)$ and an irreducible generic representation $\pi$ of $\GL_{n-1}(F)$, we introduce the notion of a ``good pair" $(\Pi,\pi)$. We point to Section \ref{section four} for a precise definition of this notion. To define the notion of a good pair, we give a condition on $\pi$ and a condition on $\Pi$. Roughly speaking, the condition on the principal series representation $\Pi$ of $\GL_{n}(F)$ is that certain cuspidal factors in the principal series do not occur in an ordering that is completely opposite to the standard ordering. Our first main result (Theorem \ref{multiplicity one for generic}) says that if ($\Pi$,$\pi$) is a good pair then $\Hom$ multiplicity one holds. As a consequence, we are able to completely determine the dimension of $\Hom_{\GL_{n-1}(F)}(\Pi,\pi)$ when $\pi$ is essentially square integrable (see Theorems \ref{multiplicity for steinberg} and \ref{multiplicity for generalized steinberg}).

We now state the main results of this paper. For $F$ a non-archimedean local field and for an integer $n\geq 0$, let $\G_n=\GL_{n}(F)$. Let $\nu(g) = |\det(g)|_F$ for $g\in \G_n$. We refer to Section \ref{section two} for more notations and terminology.

\begin{theorem} \label{multiplicity one for generic}

Let $n\geq 2$ be an integer. Let $\pi$ be an irreducible generic representation of $\G_{n-1}$. Let $\Pi$ be a principal series representation of $\G_n$ parabolically induced from an irreducible cuspidal representation of a Levi subgroup. If $(\Pi,\pi)$ is a good pair, then
\[ \dim \Hom_{\G_{n-1}}(\Pi,\pi) = 1,\] and for all integers $i\geq 1$,
\[ \Ext^i_{\G_{n-1}}(\Pi,\pi) = 0.\]

\end{theorem}

\begin{remark}
The above result is proved by sequentially replacing those cuspidal factors in $\Pi$ that occur in the cuspidal support of $\nu^{-1/2}\pi$. A combinatorial argument allows us to replace the cuspidal factors of $\Pi$ one at a time by repeated applications of Lemmas \ref{reduction lemma one} and \ref{reduction lemma two}. The proofs of Lemmas \ref{reduction lemma one} and \ref{reduction lemma two} in turn crucially depend on a filtration for parabolically induced modules introduced by K.Y. Chan in the work \cite{ch22} which we recall in Section \ref{section three}. The filtration in Lemma \ref{filtration for parabolically induced modules} stated in terms of the Fourier-Jacobi and Rankin-Selberg models is coarser than the Bernstein-Zelevinsky filtration but is much more useful for our purposes. We take advantage of the Jacquet module structure of a generic representation (with respect to the opposite parabolic subgroup) to establish Lemmas \ref{reduction lemma one} and \ref{reduction lemma two}.   
\end{remark}

The above theorem allows us to completely determine the dimension of $\Hom_{\G_{n-1}(F)}(\Pi,\pi)$ when $\pi$ is essentially square integrable. We obtain part (A) of Theorem \ref{multiplicity for steinberg} and Theorem \ref{multiplicity for generalized steinberg} as a simple consequence of the above theorem.

\begin{theorem} \label{multiplicity for steinberg}

Let $n\geq 2$ be an integer. Let $\Pi$ be a principal series representation of $\G_n$ parabolically induced from an irreducible cuspidal representation of a Levi subgroup. Let $\St_{n-1}$ denote the Steinberg representation of $\G_{n-1}$. 

\begin{itemize}
    
\item[(A)] If $\Pi \neq \nu^{-(\frac{n-1}{2})} \times 
\nu^{-(\frac{n-3}{2})} \times \ldots \times \nu^{(\frac{n-1}{2})}$, then 
\[ \dim \Hom_{\G_{n-1}}(\Pi,\St_{n-1}) = 1.\] Moreover, if $\Pi \neq \nu^{-(\frac{n-1}{2})} \times 
\nu^{-(\frac{n-3}{2})} \times \ldots \times \nu^{(\frac{n-1}{2})}$, then 
\[ \Ext^i_{\G_{n-1}}(\Pi,\St_{n-1}) = 0,\] for all integers $i\geq 1$.
               
\item[(B)] If $\Pi = \nu^{-(\frac{n-1}{2})} \times 
\nu^{-(\frac{n-3}{2})} \times \ldots \times \nu^{(\frac{n-1}{2})}$, then 
\[ \dim \Hom_{\G_{n-1}}(\Pi,\St_{n-1}) = n.\]  

\end{itemize}
\end{theorem}

The next theorem is the analogue of Theorem \ref{multiplicity for steinberg} for the generalized Steinberg representation. In contrast to Theorem \ref{multiplicity for steinberg}, in this case higher multiplicity does not occur, and $\Hom$ multiplicity one and higher $\Ext$ vanishing always holds.

\begin{theorem} \label{multiplicity for generalized steinberg}

Let $k,m \in \mathbb{Z}_{>0}$ such that $k\geq 2$ and let $n = km +1$. Let $\rho$ be a cuspidal representation of $\G_k$ and consider the segment $\Delta = [\rho,\nu^{m-1}\rho]$. Let $\Pi$ be a principal series representation of $\G_n$ parabolically induced from an irreducible cuspidal representation of a Levi subgroup. Then, 
\[ \dim \Hom_{\G_{n-1}}(\Pi,Q(\Delta)) = 1.\]
Moreover for all integers $i\geq 1$,
\[ \Ext^i_{\G_{n-1}}(\Pi,Q(\Delta)) = 0.\]

\end{theorem}

\begin{remark}
Note that just like the result of K.Y. Chan for standard representations both part (A) of Theorem \ref{multiplicity for steinberg} and Theorem \ref{multiplicity for generalized steinberg} are compatible with the Euler-Poincare pairing formula of Dipendra Prasad introduced in \cite{pr18}. If $\pi_1$ and $\pi_2$ are smooth finite-length representations of $\G_n$ and $\G_{n-1}$ respectively, we have the following Euler-Poincare pairing formula (\cite[Theorem 4.2]{pr18}):
\[ \dim \Wh(\pi_1).\dim \Wh(\pi_2) = \sum_{i\geq 0} (-1)^
{i} \dim \Ext^{i}_{\G_{n-1}}(\pi_1,\pi_2), \] where $\Wh(\pi_1)$ and $\Wh(\pi_2)$ are the spaces of Whittaker models for $\pi_1$ and $\pi_2$ respectively.    
\end{remark}

\begin{remark}
Since the multiplicity one property holds for irreducible representations (see \cite{agrs10}), we may conclude by part (B) of Theorem \ref{multiplicity for steinberg} that there are at least $n$ irreducible subquotients of $\Pi = \nu^{-(\frac{n-1}{2})} \times 
\nu^{-(\frac{n-3}{2})} \times \ldots \times \nu^{(\frac{n-1}{2})}$ that have $\St_{n-1}$ as a quotient upon restriction. These $n$ irreducible subquotients also occur in other principal series obtained by permuting the factors of $\Pi$ but in that case, their contribution does not lead to higher multiplicity as seen in part (A) of Theorem \ref{multiplicity for steinberg}. We may ask whether there are more than $n$ irreducible subquotients of $\Pi$ that have $\St_{n-1}$ as a quotient upon restriction. Our next result says that there are exactly $n$ irreducible subquotients of $\Pi$ that have $\St_{n-1}$ as a quotient upon restriction to $\GL_{n-1}(F)$.
\end{remark}

The following theorem explains the occurrence of the multiplicity $n$ in part (B) of Theorem \ref{multiplicity for steinberg}. It proves that there are exactly $n$ irreducible subquotients of $\Pi = \nu^{-(\frac{n-1}{2})} \times \nu^{-(\frac{n-3}{2})} \times \ldots \times \nu^{(\frac{n-1}{2})}$  (which has $2^{n-1}$ subquotients) which have a surjective map onto $\St_{n-1}$.

\begin{theorem} \label{list}

Let $n\geq 2$ be an integer. Let $\Pi = \nu^{-(\frac{n-1}{2})} \times \nu^{-(\frac{n-3}{2})} \times \ldots \times \nu^{(\frac{n-1}{2})}\in \Alg(\G_n)$. There are exactly $n$ irreducible subquotients $\pi_0,\pi_1,\ldots,\pi_{n-1}$ of $\Pi$ such that $\Hom_{\G_{n-1}}(\pi_i,\St_{n-1}) = \mathbb {C}$. The complete list of such irreducible subquotients in terms of the Zelevinsky classification is the following: 
\begin{itemize}
    \item $\pi_0 = Z(\Delta_1,\ldots,\Delta_n) = \St_n$ where $\Delta_j = \{\nu^{(\frac{n-2j+1}{2})}\}$ for $j=1,2,\ldots,n$.
    \item For all $i=1,2,\ldots, (n-1)$ we have $\pi_i = Z(\Delta_{i1},\Delta_{i2},\ldots,\Delta_{i(n-1)})$ where,
    
    $\Delta_{ij} = 
    \begin{cases}
    \{ \nu^{(\frac{n-2j+1}{2})} \}, & \text{ for } j=1,2,\ldots,(i-1) \\
    \{ \nu^{(\frac{n-2j-1}{2})} , \nu^{(\frac{n-2j+1}{2})} \}, & \text{ for } j=i \\
    \{ \nu^{(\frac{n-2j-1}{2})} \}, & \text{for } j=(i+1),\ldots,(n-1).
    \end{cases}
     $
\end{itemize}
Thus for all $i=1,2,\ldots,(n-1)$, $\pi_i$ is the unique irreducible submodule of the principal series $\nu^{a_i} \St_{i-1} \times \nu^{b_i}\mathbbm{1}_2 \times \nu^{c_i} \St_{n-i-1}$, where the integers $a_i$, $b_i$ and $c_i$ are, \[ a_i = \frac{(n-i+1)}{2}, b_i=\frac{(n-2i)}{2}, c_i = \frac{-(i+1)}{2}. \] Here, $\mathbbm{1}_2$ denotes the trivial representation of $\GL_2(F)$. 
    
\end{theorem}

\begin{remark}
One can write down the Langlands parameters of the irreducible representations $\pi_1,\ldots,\pi_{n-1}$ listed in the above theorem as follows. Given an irreducible representation $\tau$ of $\G_n$, let $L(\tau)$ denote its Langlands parameter. Then one can prove that, \[ L(\pi_{n-i}) = L(\nu^{-\frac{(n-i)}{2}}\St_{i}) \oplus L(\nu^{\frac{i}{2}}\St_{n-i}) \] for all $i=1,2,\ldots, (n-1)$. Note that $\pi_i = \pi_{n-i}^{\vee}$ for all $i=1,2,\ldots, (n-1)$.
\end{remark}

\begin{remark}
In order to prove that the representations $\pi_i$ of $\GL_n(F)$ listed in Theorem \ref{list} have $\St_{n-1}$ as a quotient upon restriction to $\GL_{n-1}(F)$ we show that these are the only possible irreducible subquotients of $\Pi$ which can have $\St_{n-1}$ as a quotient upon restriction. We then use the multiplicity one theorem for irreducible representations and part (B) of Theorem \ref{multiplicity for steinberg} to conclude the argument. In Section \ref{section nine}, we give a geometric explanation for why these subquotients of $\Pi$ have $\St_{n-1}$ as a quotient upon restriction.
\end{remark}

It is well known (see \cite[Theorem 3]{pr93}) that given irreducible generic representations $\pi_1$ and $\pi_2$ of $\G_{n}$ and $\G_{n-1}$ respectively, $\Hom_{\G_{n-1}}(\pi_1,\pi_2) \neq 0$. Therefore, any irreducible generic representation of $\G_n$ has $\St_{n-1}$ as a quotient upon restriction. The above result gives us some examples of non-generic irreducible representations that have $\St_{n-1}$ as a quotient upon restriction. It turns out that these are the only irreducible non-generic representations that have the Steinberg as a quotient upon restriction to $\G_{n-1}$. On the other hand, there does not exist any non-generic irreducible representation that has a generalized Steinberg of $\G_{n-1}$ as a quotient upon restriction to $\G_{n-1}$. These two claims are the content of our next theorem.

\begin{theorem} \label{list of non generic}
Let $n\geq 2$ be an integer. Let $\pi_1, \pi_2, \ldots ,\pi_{n-1}$ be the irreducible representations of $\G_n$ as defined in Theorem \ref{list}.
\begin{itemize}
    
\item[(A)] Let $\tau$ be a non-generic irreducible representation of $\G_n$ such that, $\Hom_{\G_{n-1}}(\tau,\St_{n-1})\neq 0$. Then $\tau = \pi_{i_0}$ for some $i_0 \in \{ 1, 2, \ldots,$ $(n-1) \}$.
               
\item[(B)] There does not exist any non-generic irreducible representation $\tau$ of $\G_n$ such that $\tau|_{\G_{n-1}}$ has a generalized Steinberg of $\G_{n-1}$ as a quotient.

\end{itemize}

\end{theorem}

The proof of the above theorem is obtained by combining information obtained from the left and right Bernstein-Zelevinsky filtrations. The notion of left Bernstein-Zelevinsky derivatives was introduced in the work \cite{cs21}.

In Theorem \ref{multiplicity for steinberg}, there is a sudden jump in multiplicity from $1$ to $n$. One may ask whether each of the intermediate multiplicities $2, 3, \ldots,(n-1)$ is also attained for some principal series representation. In the next theorem, we exhibit specific principal series for which each of the intermediate multiplicities is attained. Naturally many other principal series representations may also attain these intermediate multiplicities and we are only exhibiting some specific examples.

\begin{theorem} \label{intermediate multiplicities}

Let $n\geq 3$ be an integer. For each integer $i = 1, 2,3, \ldots, (n-2)$, consider the segment $\Delta_i = [\nu^{-(\frac{n-1}{2})}, \nu^{-(\frac{n-1-2i}{2})}]$. Let $\Pi_i \in \Alg(\G_n)$ denote the principal series representation,
\[ \Pi_i = Q(\Delta_i) \times \nu^{-(\frac{n-3-2i}{2})} \times \nu^{-(\frac{n-5-2i}{2})} \times \ldots \times \nu^{(\frac{n-1}{2})}. \]
Then for all $i = 1, 2, \ldots, (n-2)$,
\[ \dim \Hom_{\G_{n-1}}(\Pi_i,\St_{n-1}) = n - i. \]
    
\end{theorem}

\begin{remark}
The exact multiplicities in part (B) of Theorem \ref{multiplicity for steinberg} and in Theorem \ref{intermediate multiplicities} are proved by obtaining upper and lower bounds. The upper bound on the multiplicity is obtained using the filtration in Lemma \ref{filtration for parabolically induced modules} whereas the lower bound is obtained via Mackey theory (see \cite[Section 5]{vs13}).    
\end{remark}

We observe a curious phenomenon in Theorems \ref{multiplicity for steinberg} and \ref{multiplicity for generalized steinberg}. For a segment say $\Delta =[\rho,\nu^{m-1}\rho]$, the space $\Hom_{\G_{n-1}}(\Pi,Q(\Delta))$ can have dimension greater than 1 only if $n(\rho)=1$. One may wonder whether this phenomenon also occurs if we replace $Q(\Delta)$ with $Z(\Delta)$. As pointed out earlier, by the work of C. G. Venketasubramanian, cf. \cite{vs13}, we know that the space $\Hom_{\G_{n-1}}(\Pi,Z(\Delta))$ can have dimension 2 when $n(\rho)=1$. The content of our next theorem is that analogous to our previous observation, higher multiplicity for the space $\Hom_{\G_{n-1}}(\Pi,Z(\Delta))$ does not occur if $n(\rho)>1$. The theorem is a simple consequence of the multiplicity one theorem for standard representations, (cf. \cite{ch23}) due to K.Y. Chan.

\begin{theorem} \label{multiplicity for ZDelta}

Let $k,m \in \mathbb{Z}_{>0}$ such that $k\geq 2$ and let $n = km +1$. Let $\rho$ be a cuspidal representation of $\G_k$ and consider the segment $\Delta = [\rho,\nu^{m-1}\rho]$. Let $\Pi$ be a principal series representation of $\G_n$ parabolically induced from an irreducible representation of a Levi subgroup. Then, 
\[ \dim \Hom_{\G_{n-1}}(\Pi,Z(\Delta)) \leq 1.\]

\end{theorem}

\subsection*{Organization of the Paper} 

In Section \ref{section two}, we fix some notations used in this paper and discuss preliminaries. In Section \ref{section three}, we recall the filtration for parabolically induced modules introduced by K.Y. Chan that plays a crucial role in our arguments. In Section \ref{section four}, we define the notion of a good pair. In Section \ref{section five}, we prove Theorem \ref{multiplicity one for generic}. We prove Theorem \ref{multiplicity for steinberg}, Theorem \ref{multiplicity for generalized steinberg}, and Theorem \ref{intermediate multiplicities} in Section \ref{section six}. In Section \ref{section seven}, we prove Theorems \ref{list} and \ref{list of non generic}. We prove Theorem \ref{multiplicity for ZDelta} in Section \ref{section eight}. 

In Section \ref{section nine}, we give a geometric explanation for why the irreducible representations listed in Theorem \ref{list} have $\St_{n-1}$ as a quotient upon restriction. In Section \ref{section ten}, we give an alternative argument to prove part of Theorem \ref{multiplicity for steinberg} (B). We point out that in the last section, we study certain geometric methods for obtaining some results on branching laws and higher multiplicity which also work for pairs such as $(\SO(n),\SO(n-1))$ and $(\U(n),\U(n-1))$.

\subsection*{Acknowledgements:}The author thanks Professor Dipendra Prasad for his constant encouragement and several useful discussions. The author thanks him for a careful reading of the paper and for his comments on the paper.

\section{Preliminaries} \label{section two}

\subsection{Basic Notations}

Let $F$ be a non-archimedean local field and let $\G_n=\GL_{n}(F)$. Let $\nu(g) = |\det(g)|_F$ for $g\in \G_n$. Let $\Alg(\G_n)$ denote the category of smooth representations of $\G_n$. All representations considered in this paper are smooth. For a representation $\pi$ of $\G_n$, we set $n(\pi)=n$ and denote the smooth dual of $\pi$ by $\pi^{\vee}$. 

Given a group $G$ and a closed subgroup $H$, for $\pi\in \Alg(H)$ we let $\Ind_H^{G}(\pi)$ (resp. $\ind_H^{G}(\pi)$) denote the normalized induction (resp. normalized compact induction). For $\pi_1\in \Alg(\G_{n_1})$ and $\pi_2\in \Alg(\G_{n_2})$ we denote by $\pi_1 \times \pi_2$ the representation of $G_{n_1+n_2}$ obtained by normalized parabolic induction of the representation $\pi_1 \boxtimes \pi_2$ of $\G_{n_1}\times \G_{n_2}$.

Let $\rho$ be a cuspidal representation of $\G_n$. A segment $\Delta$ built out of $\rho$ is a set of irreducible representations of $\G_n$ of the form, 
\[ \Delta = [\nu^a\rho,\nu^b\rho] = \{\nu^a\rho,\nu^{a+1}\rho,\ldots,\nu^b\rho\} \] 
where, $a, b \in \mathbb{C}$ and $b-a\in \mathbb{Z}_{\geq 0}$. We set $a(\Delta)=\nu^a\rho$ and $b(\Delta)=\nu^b\rho$. We define the relative length (resp. absolute length) of $\Delta$ to be $l_r(\Delta)=(b-a+1)$ (resp. $l_a(\Delta)=n(b-a+1)$).  We define $\Delta^{\vee}$ to be the segment,
\[ \Delta^{\vee} = \{\nu^{-b}\rho^{\vee},\nu^{-b+1}\rho^{\vee},\ldots,\nu^{-a}\rho^{\vee}\}. \] 
We let $Q(\Delta)$ (resp. $Z(\Delta)$) denote the unique irreducible quotient (resp. submodule) of the representation, 
\[ \nu^a\rho\times \nu^{a+1}\rho\times \ldots \times \nu^b\rho.\] 
Any irreducible essentially square-integrable representation of $\G_n$ is of the form $Q(\Delta)$ (see \cite[Theorem 9.3]{ze80}). If $\Delta = [\nu^{-(\frac{n-1}{2})},\nu^{(\frac{n-1}{2})}]$, then $Q(\Delta)$ is called the Steinberg representation of $\G_n$ and is denoted as $\St_n$. Note that $\St_1$ is simply the trivial representation of $\G_1$. Given a segment $\Delta = [\nu^a\rho,\nu^b\rho]$, if $n(\rho)\geq 2$ then $Q(\Delta)$ is called a generalized Steinberg representation.

Let $\pi_i \in \Alg(\G_{n_i})$ for $i=1,2,\ldots r$ be irreducible representations. By a principal series representation of 
$\G_n$, we mean a representation of the form,
\[\Pi = \pi_1 \times \pi_2 \times \ldots \times \pi_r,\]
where $n=\sum_{i=1}^{r} n_i$. Given an irreducible representation $\pi$ of $\G_n$, there exists a unique collection ${\{\rho_1,\rho_2,\ldots,\rho_s\}}$ of cuspidal representations such that $\pi$ is a subquotient of $\rho_1\times\rho_2\times \ldots\times \rho_s$. This collection is known as the cuspidal support of $\pi$ and is denoted as $\csupp(\pi)$. We will also say that the cuspidal support of the principal series $\rho_1\times \rho_2 \times \ldots \times \rho_s$ is the multiset ${\{\rho_1,\rho_2,\ldots,\rho_s\}}$. Given cuspidal representations $\rho_1$ and $\rho_2$ of $\G_n$, if $\rho_1=\nu^c\rho_2$ for some integer $c$, we say that $\rho_1$ and $\rho_2$ lie in the same cuspidal line. Given an irreducible representation $\pi$ of $\G_n$, we let $\csupp_{\mathbb{Z}}(\pi)$ denote the multiset ${\{\nu^c\rho : c \in \mathbb{Z}, \rho \in \csupp(\pi) \}}$.

We say that two segments $\Delta_1$ and $\Delta_2$ are linked if $\Delta_1 \not \subset \Delta_2$, $\Delta_2 \not \subset \Delta_1$ and $\Delta_1 \cup \Delta_2$ is a segment. We say that $\Delta_1$ precedes $\Delta_2$ if $\Delta_1$ and $\Delta_2$ are linked and $b(\Delta_2)=\nu^m b(\Delta_1)$ for some integer $m>0$. Suppose that we are given a multiset of segments (called a multisegment) $\mathfrak{m} = \{ \Delta_1,\Delta_2,\ldots,\Delta_r \}$ such that $\Delta_i$ does not precede $\Delta_j$ for $i<j$. By the Zelevinsky classification, the representation $Z(\Delta_1)\times Z(\Delta_2) \times \ldots \times Z(\Delta_r)$ has a unique irreducible submodule which we denote as $Z(\Delta_1,\Delta_2,\ldots,\Delta_r)$. Similarly by the Langlands classification, the representation $Q(\Delta_1)\times Q(\Delta_2) \times \ldots \times Q(\Delta_r)$ has a unique irreducible quotient which we denote as $Q(\Delta_1,\Delta_2,\ldots,\Delta_r)$. A representation of the form $Q(\Delta_1)\times Q(\Delta_2) \times \ldots \times Q(\Delta_r)$ such that $\Delta_i$ does not precede $\Delta_j$ for $i<j$ is called a standard representation.

Let $U_n$ denote the group of upper triangular unipotent matrices of $\G_n$. Let $\psi_n$ be a non-degenerate character of $U_n$. The representation $\ind_{U_n}^{\G_n}\psi_n$ is known as the Gelfand-Graev representation of $\G_n$. 

\subsection{The Gelfand-Kazhdan Automorphism}

The automorphism $\theta(g) = (g^{-1})^t$ of $\G_n$ induces an exact functor on $\Alg(\G_n)$. This functor on $\Alg(\G_n)$, once again denoted as $\theta$, is called the Gelfand-Kazhdan automorphism.

Let $\pi_1$ and $\pi_2$ be smooth representations of $\G_n$ and $\G_{n-1}$  respectively. Since the automorphism $\theta$ preserves $\G_{n-1}$, for all $i\geq 0$, 
\[ \Ext^i_{\G_{n-1}}(\pi_1,\pi_2) \cong \Ext^i_{\G_{n-1}}(\theta(\pi_1),\theta(\pi_2)). \]
The following result is due to Gelfand and Kazhdan (see \cite{gk75}).
\begin{lemma} \label{theta is dual}

If $\pi \in \Alg(\G_n)$ is irreducible, then $\theta(\pi) \cong \pi^{\vee}$.

\end{lemma}
The following lemma (\cite[Lemma 1.9]{ze80}) tells us that $\theta$ acts on a product of two representations by reversing their order.
\begin{lemma} \label{theta on principal series}
Let $\pi_1$ and $\pi_2$ be smooth representations of $\G_{n_1}$ and $\G_{n_2}$  respectively. Then, 
\[ \theta(\pi_1 \times \pi_2)\cong \theta(\pi_2) \times \theta(\pi_1). \]

\end{lemma}

\subsection{Left and Right Derivatives}

Given $\pi\in \Alg(\G_n)$, we define the $i$-th Bernstein-Zelevinsky derivative $\pi^{(i)}$ as follows. Let $U_i\subset \G_i$ be the subgroup of upper triangular unipotent matrices. Fix a non-degenerate character $\psi_i$ of $U_i$. Let $R_{n-i}$ be the subgroup of $\G_n$ defined as,
\[ R_{n-i} = \left\{ \begin{pmatrix}
I_{n-i} & v \\
0  & z \\
\end{pmatrix} : v \in M_{(n-i) \times i}, z\in U_i \right \}. \]
We abuse notation to denote by $\psi_i$ the character of $R_{n-i}$ defined as,
\[ \psi_i(\begin{pmatrix}
I_{n-i} & v \\
0  & z \\
\end{pmatrix} ) =  \psi_i^{\prime}(z). \]
Let $\delta$ be the modular character of $R_{n-i}$. Then $\pi^{(i)}$ is defined as the normalized twisted Jacquet module,
\[ \pi^{(i)} = \delta^{-1/2} \pi / \langle \pi(u).x - \psi_i(u).x : x \in \pi, u\in R_{n-i}
 \rangle. \] 
 
We shall call the usual $i$-th Bernstein-Zelevinsky derivative $\pi^{(i)}$ as the right $i$-th Bernstein-Zelevinsky derivative. The notion of left Bernstein-Zelevinsky derivative was introduced in the work of K.Y. Chan and G. Savin (see \cite{cs21}). We define the left $i$-th Bernstein-Zelevinsky  derivative $^{(i)}\pi$ as,
\[ ^{(i)}\pi = \theta(\theta(\pi)^{(i)}) .\] 
We recall that by the Bernstein-Zelevinsky filtration (see \cite{bz76}), we have a filtration of submodules $ 0=V_n\subsetneq V_{n-1}\subsetneq \ldots \subsetneq V_0 = \pi|_{\G_{n-1}}$ such that,
\[ V_i/V_{i+1} \cong \ind_{\G_{n-i-1}R_{n-i-1}}^{\G_{n-1}}(\nu^{1/2}\pi^{(i+1)}\otimes \psi_i). \]

We shall frequently make use of the following product rule for derivatives (see \cite[Corollary 4.14(c)]{bz77}).
\begin{lemma}
Let $\pi_j \in \Alg(\G_{n_j})$ for $j=1,2,\ldots,k$. Then the representation $(\pi_1 \times \pi_2 \times \ldots \times \pi_k)^{(i)}$ has a filtration whose successive quotients are, \[\pi_1^{(i_1)} \times \pi_2^{(i_2)} \times \ldots \times \pi_k^{(i_k)},\] where $0\leq i_t \leq n_t,$ for $t = 1, 2, 3, \ldots, k$ and $i_1+i_2+\ldots+i_k = i$.   
\end{lemma}
There is a similar product rule for left derivatives. We now collect some lemmas that we shall need. 

\begin{lemma}(see \cite[Proof of Lemma 3.2]{cs21}) \label{left right bz filtration}
Let $\pi_1\in\Alg(\G_n)$ and $\pi_2\in\Alg(\G_{n-1})$. Then,
\begin{itemize}
\item[(A)] $\dim \Hom_{\G_{n-1}}(\pi_1,\pi_2) \leq \sum_{i=1}^{n}\dim \Hom_{G_{n-i}}(\nu^{1/2}.\pi_1^{(i)},^{(i-1)}\pi_2).$
\item[(B)] $\dim \Hom_{\G_{n-1}}(\pi_1,\pi_2) \leq \sum_{i=1}^{n}\dim \Hom_{G_{n-i}}(\nu^{-1/2}.^{(i)}\pi_1,\pi_2^{(i-1)}).$
\end{itemize}
\end{lemma}

Let $\rho$ be a cuspidal representation of $\G_r$. Given a segment $\Delta=[\nu^a\rho,\nu^b\rho]$ we let $\Delta^{(k)}$ (resp. $^{(k)}\Delta$) be the segment obtained from $\Delta$ by truncating from the right (resp. left) $k$ times. We denote $\Delta^{-} = \Delta^{(1)}$ and $^{-}\Delta=$ $^{(1)}\Delta$. We state the next two lemmas with this notation. 

\begin{lemma}(see \cite[Proposition 2.7]{cs21}) Let $i\in \mathbb{Z}_{>0}$. Then the $i$-th left and right derivatives of $Q(\Delta)$ vanish unless $i=jr$ for some $j\in \mathbb{Z}_{\geq 0}$ in which case, \[ Q(\Delta)^{(i)} = Q(^{(j)} \Delta) \text{ and } ^{(i)}Q(\Delta) = Q( \Delta^{(j)}). \]
\end{lemma}
\begin{lemma}(see \cite[Proposition 2.5]{cs21}) \label{derivative of ZDelta} Let $i\in \mathbb{Z}_{>0}$. Then the $i$-th left and right derivatives of $Z(\Delta)$ vanish unless $i=r$ in which case, \[ Z(\Delta)^{(i)} = Z( \Delta^{-}) \text{ and } ^{(i)}Z(\Delta) = Z(^{-}\Delta). \]
\end{lemma}

\subsection{A Second Adjointness Formula}

Let $r_{(n-l,l)}$ (resp. $\bar{r}_{(n-l,l)}$) denote the Jacquet functor on the category $\Alg(\G_n)$ with respect to the parabolic subgroup (resp. opposite parabolic subgroup) corresponding to the partition $(n-l,l)$ of $n$.

\begin{lemma}(see \cite[Proposition 9.5]{ze80})
 Let $\rho$ be a cuspidal representation of $\G_m$ and consider the segment $\Delta=[\rho,\rho^{k-1}\rho]$ where, $n=km$. If $l$ is not divisible by $m$ then $r_{(n-l,l)}(Q(\Delta))=0$. If $l=mp$ then, \[ r_{(n-l,l)}(Q(\Delta))= Q([\nu^p\rho,\rho^{k-1}\rho]) \otimes Q([\rho,\nu^{p-1}\rho]). \]
\end{lemma}

We shall require the following formula for the Jacquet module with respect to the opposite parabolic subgroup.
\begin{lemma} \label{second adjointness formula}
 Let the notation be as in the previous lemma. If $l$ is not divisible by $m$ then $\bar{r}_{(n-l,l)}(Q(\Delta))=0$. If $l=mp$ then, \[ \bar{r}_{(n-l,l)}(Q(\Delta))= Q([\rho,\nu^{k-p-1}\rho]) \otimes Q([\nu^{k-p}\rho,\nu^{k-1}\rho]). \]
\end{lemma}

\section{Filtration for Parabolically Induced Modules} \label{section three}

In this section we recall the filtration on parabolically induced modules introduced by K.Y. Chan in the work \cite{ch22}. This filtration on a principal series of the form $\pi_1\times \pi_2$ is coarser than the Bernstein-Zelevinsky filtration but has the advantage that it allows us to maintain control on the second factor while we apply the Bernstein-Zelevinsky filtration on the first factor. This will be useful for us in the proof of Lemmas \ref{reduction lemma one} and \ref{reduction lemma two}.

We first recall the (equal rank) Fourier-Jacobi Model and Rankin-Selberg Models introduced by K.Y. Chan in the work \cite{ch22}. They arise in dealing with the restriction problem from $\GL_n(F)$ to $\GL_{n-1}(F)$ of a parabolically induced representation of $\GL_n$. Such reductions are at the basis of the Bessel models of GGP, cf. \cite[Theorem 15.1]{ggp11}.

\subsection{Fourier-Jacobi Model (see \cite[Section 2.5]{ch23})}

Let $S(F^n)$ denote the space of Schwartz functions on $F^n$ (locally constant functions with compact support). Define $\zeta^F$ to be the representation of $\G_n$ with underlying space $S(F^n)$ where the group action is defined as follows. For $g\in \G_n$ and $f\in S(F^n)$,
\[ (g.f)(x) = \nu^{-1/2}(g)f(g^{-1}.x). \]
For $\pi\in \Alg(\G_n)$, the representation $\pi\otimes \zeta^F$ of $\G_n$ is known as the equal rank Fourier-Jacobi model of $\pi$.

\subsection{Rankin-Selberg Model (see \cite[Section 2.4]{ch23})}
Let $U_r$ denote the group of upper triangular unipotent matrices of $\G_r$. For $r\geq 0$, consider the subgroup $H_{r,n}^R$ of $\G_{n+1}$ defined as, 
\[ H_{r,n}^R = \left\{ \begin{pmatrix}
g & 0 & x\\
  & 1 & v^t \\
  &   & u\\
\end{pmatrix} : g\in \G_{n-r}, x\in M_{n-r,r}, u\in U_r, v\in F^r \right \}. \] 
Define the character $\zeta_R : H_{r,n}^R \rightarrow \mathbb{C}$ as,
\[ \zeta_R(\begin{pmatrix}
g & 0 & x\\
  & 1 & v^t \\
  &   & u\\
\end{pmatrix}) = \psi(\begin{pmatrix}
1 & v^t \\
  & u\\
\end{pmatrix}) \]
where $\psi$ is a non-degenerate character on $U_{r+1}$. 

Given a smooth representation $\pi$ of $\G_{n-r}$, we trivially extend it to a representation of $H_{r,n}^R$. Then the Rankin-Selberg Model of $\pi$ is defined as, \[ \RS_r(\pi) = \ind_{H_{r,n}^R}^{\G_{n+1}} (\pi \otimes \zeta_R). \]
We point out that if $\pi\in \Alg(\G_n)$ then $\RS_0(\pi) = \ind_{\G_n}^{\G_{n+1}}\pi.$

\subsection{A Filtration for Parabolically Induced Modules}

We have the following filtration for parabolically induced modules stated in terms of the Fourier-Jacobi and Rankin-Selberg Models (see \cite[Proposition 2.4]{ch23}).

\begin{lemma} \label{filtration for parabolically induced modules}

Let $\pi_1$ and $\pi_2$ be smooth representations of $\G_{n_1}$ and $\G_{n_2}$ respectively. Then $(\pi_1\times\pi_2)|_{\G_{n_1+n_2-1}}$ has a filtration,
\[0=V_d\subsetneq V_{d-1}\subsetneq \ldots \subsetneq V_0 = (\pi_1\times\pi_2)|_{\G_{n_1+n_2-1}}\]
such that, $$ V_0/V_1 \cong \nu^{1/2}\pi_1\times\pi_2|_{\G_{n_2-1}}, $$ $$ V_1/V_2 \cong \nu^{1/2}\pi_1^{(1)}\times(\pi_2\otimes \zeta^F)$$ and for $2\leq k\leq d-1$, 
$$V_k/V_{k+1} \cong \nu^{1/2}\pi_1^{(k)}\times\RS_{k-2}(\pi_2),$$
where, $\pi_2\otimes \zeta^F$ is the Fourier-Jacobi model of $\pi_2$ and $\RS_{k-2}(\pi_2)$ is the Rankin-Selberg model of $\pi_2$. 
\end{lemma}

\section{Definition of a Good Pair} \label{section four}

Given a principal series representation of $\G_n$ induced from a cuspidal representation of a Levi subgroup and an irreducible generic representation of $\G_s$, where $s\leq n$, we wish to define the notion of when they form a good pair. We first set up the notations required for our definition.

Let $\Pi$ be a principal series representation of $\G_n$ induced from a cuspidal representation of a Levi subgroup. Suppose $\Pi=\rho_1\times \rho_2 \times \ldots \times \rho_k$ where, $\rho_1,\rho_2, \ldots,\rho_k$ are cuspidal representations of $\G_{n_i}$ ($i=1,2,\ldots,k$) such that $\sum_{i=1}^{k} n_i = n$. We first define the notion of when $\Pi$ is said to be good to a given segment $\Delta$.

\begin{definition} Let $\Pi$ be as above. Consider the segment $\Delta=[\tau,\nu^c\tau]$ where $\tau$ is a cuspidal representation of $\G_{m}$ (for some $m \in \mathbb{Z}_{>0}$) and $c$ is a nonnegative integer. We say that the principal series $\Pi$ is bad to $\Delta$ if there exists a strictly increasing sequence $s_0<s_1< \ldots <s_{c}<s_{c + 1}$ of natural numbers such that, \[ \rho_{s_t} = \nu^{t - \frac{1}{2}} \tau \] for all $t=0,1,\ldots,c, (c + 1)$.

If $\Pi$ is not bad to $\Delta$ then we say that $\Pi$ is good to $\Delta$.
    
\end{definition}

The above definition basically says that if $\Pi$ is bad to $\Delta$ then the factors $\nu^{-\frac{1}{2}}\tau,\nu^{\frac{1}{2}} \tau,\ldots$, $\nu^{c-\frac{1}{2}} \tau, \nu^{c+\frac{1}{2}} \tau$ occur in this order (with gaps allowed in between) as some factors of the principal series $\Pi$. Notice that this ordering is completely opposite to the standard ordering. 

We now define the notion of a good pair. Let $\pi$ be an irreducible generic representation of $\G_s$ for some integer $s>0$. By Zelevinsky (see \cite{ze80}) there exist segments $\Delta_1,\Delta_2,\ldots,\Delta_r$ such that $\pi = Q(\Delta_1)\times Q(\Delta_2)\times \ldots \times Q(\Delta_r)$. Naturally, no two of these segments are linked to each other. Suppose that $\Delta_i=[\tau_i,\nu^{c_i}\tau_i]$ for all $i=1,2,\ldots,r$. Here each $\tau_i$ is a cuspidal representation of some $\G_{m_i}$ and each $c_i$ is a nonnegative integer. 

We now define when $(\Pi,\pi)$ are said to form a good pair.

\begin{definition} \label{definition of good pair}

Let $\Pi$, $\pi$ and $\Delta_i$ ($i=1,2,\ldots,r$) be as above. We say that $(\Pi,\pi)$ is a good pair if the following conditions hold:
\begin{itemize}    
\item[(A)] For all $i\neq j$, $\tau_i$ and $\tau_j$ do not lie on the same cuspidal line.               
\item[(B)] The principal series $\Pi$ is good to $\Delta_j$ for all $j=1,2,\ldots,r$.
\end{itemize}    
\end{definition}

The first condition above is a condition on the generic representation $\pi$ and the second condition is a combinatorial condition on the factors of the principal series $\Pi$. If $(\Pi,\pi)$ is a good pair we say that $\Pi$ is good to $\pi$. 

\section{Proof of Theorem \ref{multiplicity one for generic}} \label{section five}
    
We first prove two key reduction lemmas using the filtration in Lemma \ref{filtration for parabolically induced modules}. These reduction lemmas allow us to replace either the first or last factor in the principal series $\Pi$.

\subsection{Two Reduction Lemmas}

\begin{lemma} \label{reduction lemma one}
Let $n\geq 2$ be an integer. Let $\Pi=\rho_1\times \rho_2 \times \ldots \times \rho_k \in \Alg(\G_n)$ where $\rho_1,\rho_2, \ldots,\rho_k$ are cuspidal representations of $\G_{n_i}$ ($i=1,2,\ldots,k$) such that $\sum_{i=1}^{k} n_i = n$. Let $\pi$ be an irreducible generic representation of $\G_{n-1}$. Suppose that $\pi = Q(\Delta_1)\times Q(\Delta_2)\times \ldots \times Q(\Delta_r)$ for some segments $\Delta_1,\Delta_2,\ldots,\Delta_r$ such that no two of these segments are linked to each other. Let $\Delta_i=[\tau_i,\nu^{c_i}\tau_i]$ for all $i=1,2,\ldots,r$. Here each $\tau_i$ is a cuspidal representation of some $\G_{m_i}$ and each $c_i$ is a nonnegative integer.

If $\rho_1 \neq \nu^{-1/2}a(\Delta_i)$ for all $i=1,2,\ldots,r$, then there exists a cuspidal representation $\rho_1^{\prime}\in \Alg(\G_{n_1})$ such that $\rho_1^{\prime}\not \in \csupp_{\mathbb{Z}}(\Pi)\cup \csupp_{\mathbb{Z}}(\nu^{-1/2}\pi)$ and,
\[ \Ext^j_{\G_{n-1}}(\Pi, \pi) = \Ext^j_{\G_{n-1}}(\rho_1^{\prime} \times \rho_2 \times \rho_3 \times \ldots \times \rho_k, \pi) \] for all integers $j\geq 0$.
    
\end{lemma}

\begin{proof}
Let $\Pi^{\prime} = \rho_2 \times \rho_3 \times \ldots \times \rho_k$. Let us suppose that $n_1 \geq 2$. By Lemma \ref{filtration for parabolically induced modules} we have the short exact sequence,
\[ 0 \rightarrow \RS_{n_1-2}(\Pi^{\prime}) \rightarrow \Pi|_{\G_{n-1}} \rightarrow \nu^{1/2}\rho_1 \times (\Pi^{\prime}|_{\G_{n-n_1-1}}) \rightarrow 0. \]
We claim that for all integers $j\geq 0$, \[ \Ext^j_{\G_{n-1}}(\nu^{1/2}\rho_1 \times (\Pi^{\prime}|_{\G_{n-n_1-1}}), \pi) = 0.\] This is because by second adjointness,
 \[ \Ext^j_{\G_{n-1}}(\nu^{1/2}\rho_1 \times (\Pi^{\prime}|_{\G_{n-n_1-1}}), \pi) = \Ext^j_{\G_{n-1}}(\nu^{1/2}\rho_1 \otimes (\Pi^{\prime}|_{\G_{n-n_1-1}}), \bar{r}_{(n_1,n-n_1)}(\pi)),\] where $\bar{r}_{(n_1,n-1-n_1)}$ denotes the Jacquet functor with respect to the opposite parabolic subgroup corresponding to the partition $(n_1,n-1-n_1)$.

By Lemma \ref{second adjointness formula} and the geometric lemma any term in the semi-simplification of $\bar{r}_{(n_1,n-1-n_1)}(\pi)$ is of the form, \[ Q([a(\Delta_{i_1}),\nu^{d_1} a(\Delta_{i_1})]) \times \ldots \times Q([a(\Delta_{i_t}),\nu^{d_t} a(\Delta_{i_t})]) \otimes \omega \] where, $d_{i_1},d_{i_2},\ldots,d_{i_t}$ are some nonnegative integers and $\omega$ is a representation of $\G_{n-1-n_1}$. Comparing cuspidal supports at $a(\Delta_{i_1})$ proves our claim. By a long exact sequence argument, we conclude that \[ \Ext^j_{\G_{n-1}}(\Pi, \pi) = \Ext^j_{\G_{n-1}}(\RS_{n_1-2}(\Pi^{\prime}), \pi) \] for all integers $j\geq 0$.

Now choose a $\rho_1^{\prime}\in \Alg(\G_{n_1})$ such that $\rho_1^{\prime}\not \in \csupp_{\mathbb{Z}}(\Pi)\cup \csupp_{\mathbb{Z}}(\nu^{-1/2}\pi)$. Repeating the entire argument for the principal series $\rho_1^{\prime} \times \rho_2 \times \rho_3 \times \ldots \times \rho_k$ instead of $\Pi$ shows that, \[ \Ext^j_{\G_{n-1}}(\rho_1^{\prime} \times \rho_2 \times \rho_3 \times \ldots \times \rho_k, \pi) = \Ext^j_{\G_{n-1}}(\RS_{n_1-2}(\Pi^{\prime}), \pi) \] for all integers $j\geq 0$. 

By the above two equations conclude that, \[ \Ext^j_{\G_{n-1}}(\Pi, \pi) = \Ext^j_{\G_{n-1}}(\rho_1^{\prime} \times \rho_2 \times \rho_3 \times \ldots \times \rho_k, \pi) \] for all integers $j\geq 0$. If $n_1=1$ we use the equal rank Fourier-Jacobi model instead of the Rankin-Selberg model throughout the proof.

\end{proof}

\begin{lemma} \label{reduction lemma two}
Let $\Pi$, $\pi$ and $\Delta_i$ ($i=1,2,\ldots,r$) be as in Lemma \ref{reduction lemma one}.  If $\rho_k \neq \nu^{1/2}b(\Delta_i)$ for all $i=1,2,\ldots,r$, then there exists a cuspidal representation $\rho_k^{\prime}\in \Alg(\G_{n_k})$ such that $\rho_k^{\prime}\not \in \csupp_{\mathbb{Z}}(\Pi)\cup \csupp_{\mathbb{Z}}(\nu^{-1/2}\pi)$ and,
\[ \Ext^j_{\G_{n-1}}(\Pi, \pi) = \Ext^j_{\G_{n-1}}(\rho_1 \times \rho_2 \times\ldots \times \rho_{k-1} \times \rho_k^{\prime}, \pi) \] for all integers $j\geq 0$.  
\end{lemma}

\begin{proof}
By the Gelfand-Kazhdan automorphism,
\[ \Ext^j_{\G_{n-1}}(\Pi, \pi) = \Ext^j_{\G_{n-1}}(\rho_k^{\vee} \times \rho_{k-1}^{\vee} \times\ldots \times \rho_1^{\vee}, Q(\Delta_1^{\vee})\times Q(\Delta_2^{\vee})\times \ldots \times Q(\Delta_r^{\vee})).\] Note that $\rho_k^{\vee} \neq \nu^{-1/2}a(\Delta_i^{\vee})$ for all $i=1,2,\ldots,r$. Choose a cuspidal representation $\rho_k^{\prime}\in \Alg(\G_{n_k})$ such that $(\rho_k^{\prime})^{\vee}\not \in \csupp_{\mathbb{Z}}(\Pi)\cup \csupp_{\mathbb{Z}}(\nu^{-1/2}\pi) \cup \csupp_{\mathbb{Z}}(\nu^{1/2}\pi)$. Reasoning exactly as in Lemma \ref{reduction lemma one} gives that, \[ \Ext^j_{\G_{n-1}}(\Pi, \pi) = \Ext^j_{\G_{n-1}}((\rho_k^{\prime})^{\vee} \times \rho_{k-1}^{\vee} \times\ldots \times \rho_1^{\vee}, \pi^{\vee}).\] Applying the Gelfand-Kazhdan automorphism once again proves our result.

\end{proof}

\subsection{A Combinatorial Lemma}

The combinatorial lemma in this subsection allows us to rearrange the factors of certain principal series representations so that specified cuspidal representations do not occur as the first or last factors in the given principal series.

\begin{lemma} \label{a combinatorial lemma}

Let $n, k, c \in \mathbb{Z}_{>0}$ and suppose that $n\geq 2$. Let $\rho$ be a cuspidal representation of $\G_k$. Consider the segment $\Delta=[\rho,\nu^c\rho]$. Consider the principal series representation $\Pi = \pi_1\times \pi_2 \times \ldots \times \pi_n \in \Alg(\G_{nk})$ where, $\pi_1,\pi_2,\ldots,\pi_n$ are cuspidal representations of $\G_k$ lying in the same cuspidal line as $\nu^{-1/2}\rho$. Suppose that $\Pi$ is good to $\Delta$.

Then there exists a permutation $\sigma\in S_n$ such that $\Pi \cong \pi_{\sigma(1)}\times \pi_{\sigma(2)} \times \ldots \times \pi_{\sigma(n)}$ and at least one of the following conditions hold:
\begin{itemize}
\item[(A)] $\pi_{\sigma(1)}\neq \nu^{-\frac{1}{2}}\rho$.
\item[(B)] $\pi_{\sigma(n)}\neq \nu^{\frac{1}{2} + c}\rho$.
\end{itemize}
\end{lemma}

\begin{proof}
 We set $S(n) = \{ \pi_1,\pi_2,\ldots,\pi_n \}$ and $S_0 = \{ \nu^{-\frac{1}{2}}\rho,\nu^{\frac{1}{2}} \rho,\ldots,\nu^{\frac{1}{2} + c}\rho \}$. When $n=2$ the lemma is clearly true so we assume that $n\geq 3$. If either $\pi_1 \neq \nu^{-\frac{1}{2}}\rho$ or $\pi_n \neq \nu^{\frac{1}{2}+c}\rho$ then there is nothing to prove. Therefore throughout the proof, we assume that $\pi_1 = \nu^{-\frac{1}{2}}\rho$ and $\pi_n = \nu^{\frac{1}{2}+c}\rho$. If $\pi_1 = \nu^{-\frac{1}{2}}\rho$ and $\pi_n = \nu^{\frac{1}{2}+c}\rho$ and $c=0$ then $\Pi$ is bad to $\Delta$. This contradiction implies that the lemma is true for $c=0$. So from now on we assume that $c\geq 1$. We divide the proof of the lemma into two cases depending on whether $S_0\subset S(n)$ or $S_0\not \subset S(n)$. 

\textbf{Case 1} - Suppose that $S_0\not \subset S(n)$.

Then there exists an $i_0 \in \{ 1, 2,\ldots,c \}$ such that $\nu^{-\frac{1}{2} + i_0}\rho \not \in S(n)$. Consider the following two disjoint subsets of $S(n)$. 
\[ T_1 = \{ \pi \in S(n) | \pi = \nu^{-\frac{1}{2}+j} \rho \text{ for some } j < i_0, \text{ such that } j-i_0\in \mathbb{Z} \}. \]
\[ T_2 = \{ \pi \in S(n) | \pi = \nu^{-\frac{1}{2}+j} \rho \text{ for some } j > i_0, \text{ such that } j-i_0\in \mathbb{Z} \}. \]
Note that $S(n) = T_1\cup T_2$. Also, no element of $T_1$ is linked to an element of $T_2$ and versa. Note that $\pi_1 = \nu^{-\frac{1}{2}}\rho$ lies in $T_1$ and therefore $T_1$ is non-empty. Let $M$ be the largest integer such that $\pi_{M}\in T_1$. If $M=n$ then we are done since condition (B) is satisfied. If $M<n$ then since $\pi_{M}$ is not linked to $\pi_j$ for $j>M$ we can successively switch the positions of $\pi_{M}$ with $\pi_{M + 1},\pi_{M + 2},\ldots$ so on until $\pi_n$. The new principal series that we obtain looks like,
\[\Pi \cong \pi_1 \times \pi_2 \times \ldots \times \pi_{M-1}\times \pi_{M + 1} \times \pi_{M + 2}\times \ldots \times \pi_n \times \pi_{M}.\] Now, condition (B) is clearly satisfied and we are done.

Note that in this case, we can also ensure that condition (A) is satisfied by considering the smallest integer $M'$  integer such that $\pi_{M'}\in T_2$. We then switch $\pi_{M'}$ with the factors to its left until it occupies the first position.

\textbf{Case 2} - Suppose that $S_0 \subset S(n)$. For $j= 1,2,\ldots,c$ let $k_j$ be the smallest integer such that $\pi_{k_j}=\nu^{-\frac{1}{2}+j}\rho$. If $k_1 < k_2 < \ldots < k_{c}$ then the presence of the sequence $\pi_{1}, \pi_{k_1},\ldots,\pi_{k_c},\pi_{k_n}$ implies that $\Pi$ is bad to $\Delta$ which is a contradiction. Therefore there exists an integer $i$ such that $k_i > k_{i+1}$. Let $s$ be the smallest integer such integer. We now consider two sub-cases.

\textbf{Case 2(A)} - Suppose that $\pi_{k_{s+1}}\neq \nu^{-\frac{1}{2}+c}\rho$. By the minimality of $s$, we conclude that $\pi_{k_{s+1}}$ is not linked to any of the factors to its left. By consecutively switching the position of $\pi_{k_{s+1}}$ with the factor to its left we can ensure that it occupies the first position and condition (A) is satisfied.

\textbf{Case 2(B)} - Suppose that $\pi_{k_{s+1}}= \nu^{-\frac{1}{2}+c}\rho$. This means that $\pi_{k_{s}}= \nu^{-\frac{3}{2}+c}\rho$. Consider the following set. 
\[ T = \{ \pi_j \in S(n) | j\geq k_{s+1} \text{ and } \pi_j = \nu^{-\frac{1}{2}+j} \rho \text{ for some } j \leq (c-1) \}. \] Since, $\pi_{k_{s}}\in T$ we conclude that $T$ is non-empty. Let $M$ be the largest integer such that $\pi_{M}\in T$. Note that none of $\pi_{M+1},\pi_{M+2},\ldots,\pi_{n-1}$ is equal to $\nu^{-\frac{1}{2}+c}\rho$. This is because if $\pi_{M'}=\nu^{-\frac{1}{2}+c}\rho$ for some integer $(M+1)\leq M' \leq (n-1)$ then the presence of the sequence $\pi_{1}, \pi_{k_1},\ldots,\pi_{k_{s-1}},\pi_{k_{s}},\pi_{M'},\pi_{k_n}$ implies that $\Pi$ is bad to $\Delta$.

Now by the maximality of $M$, we conclude that $\pi_M$ is not linked to any factor to its right. By consecutively switching the positions of $\pi_{M}$ with $\pi_{M + 1},$ $\pi_{M + 2},\ldots$ so on until $\pi_n$ we can ensure that condition (B) is satisfied.
\end{proof}

\subsection{Proof of Theorem \ref{multiplicity one for generic}}

We recall the statement of the main theorem.

\begin{theorem} \label{multiplicity one for generic version two}

Let $n\geq 2$ be an integer. Let $\pi$ be an irreducible generic representation of $\G_{n-1}$. Let $\Pi$ be a principal series representation of $\G_n$ parabolically induced from an irreducible cuspidal representation of a proper Levi subgroup. If $\Pi$ is good to $\pi$, then
\[ \dim \Hom_{\G_{n-1}}(\Pi,\pi) = 1,\] and for all integers $i\geq 1$,
\[ \Ext^i_{\G_{n-1}}(\Pi,\pi) = 0.\]

\end{theorem}

\begin{proof}

Let $N$ denote the cardinality of the set $\csupp_{\mathbb{Z}}(\nu^{1/2}\Pi)\cap \csupp_{\mathbb{Z}}(\pi)$. The proof is via induction on $N$. When $N=0$, we use the Bernstein- Zelevinsky filtration. The pieces other than the bottom piece in the Bernstein-Zelevinsky filtration do not contribute to $\Ext^{i}(\Pi,\pi)$. This is because for $0<j<n$, we have that $\Ext^i_{G_{n-1}}(\nu^{1/2}\Pi^{(j)},$$^{(j-1)}\pi) = 0$ by comparing cuspidal supports. (Here we are using the second adjointness result in \cite[Lemma 2.1]{cs21} for the contribution of the $j$-th piece). The bottom piece of the filtration is the Gelfand-Graev representation. Since the Gelfand-Graev representation is projective (see \cite{cs19}) we conclude that $\Ext^i_{\G_{n-1}}(\Pi,\pi) = 0$ for all integers $i\geq 1$. Since the space of Whittaker functionals of $\pi$ is one dimensional we conclude that the space  $\Hom_{\G_{n-1}}(\Pi,\pi)$ has dimension 1.

Now suppose that $N>0$. Let $\Pi=\rho_1\times \rho_2 \times \ldots \times \rho_k$ where, $\rho_1,\rho_2, \ldots,\rho_k$ are cuspidal representations of $\G_{n_i}$ ($i=1,2,\ldots,k$) such that $\sum_{i=1}^{k} n_i = n$. Since $\pi$ is generic, by Zelevinsky (\cite{ze80}) there exist segments $\Delta_1,\Delta_2,\ldots,\Delta_r$ such that $\pi = Q(\Delta_1)\times Q(\Delta_2)\times \ldots \times Q(\Delta_r)$ and no two of these segments are linked to each other. Let $\Delta_i=[\tau_i,\nu^{c_i}\tau_i]$ for all $i=1,2,\ldots,r$. Here each $\tau_i$ is a cuspidal representation of some $\G_{m_i}$ and each $c_i$ is a nonnegative integer.

Let $j$ be the smallest integer such that $\rho_j$ lies in the same cuspidal line as $\nu^{-1/2}a(\Delta_s)$ for some $s\in \{1,2,\ldots,r\}$. By switching factors, we may assume that $j=1$. Also by switching factors, we can assume that for some integer 
$t\leq k$, $\rho_1,\rho_2,\ldots,\rho_t$ are the set of all cuspidal factors of $\Pi$ lying in the same cuspidal line as $\nu^{-1/2}a(\Delta_s)$. 

Since $\Pi$ is good to $\Delta_s$, the principal series $\rho_1\times\rho_2\times\ldots\times \rho_t$ is also good to $\Delta_s$. By Lemma \ref{a combinatorial lemma}, we conclude that there exists a permutation $\sigma\in S_t$ such that, $\Pi \cong \rho_{\sigma(1)}\times \rho_{\sigma(2)} \times \ldots \times \rho_{\sigma(t)}\times \rho_{t+1} \times \ldots \times \rho_k$ and either $\rho_{\sigma(1)} \neq \nu^{-1/2}a(\Delta_s)$ or $\rho_{\sigma(t)} \neq \nu^{1/2}b(\Delta_s)$.

Suppose that $\rho_{\sigma(1)} \neq \nu^{-1/2}a(\Delta_s)$. By condition (B) of Definition \ref{definition of good pair}, $\rho_{\sigma(1)} \neq \nu^{-1/2}a(\Delta_i)$ for all $i=1,2,\ldots,r$. We now invoke Lemma \ref{reduction lemma one} to conclude that, \[ \Ext^j_{\G_{n-1}}(\Pi, \pi) = \Ext^j_{\G_{n-1}}(\rho_1^{\prime} \times \rho_2 \times \rho_3 \times \ldots \times \rho_k, \pi) \] for some $\rho_1^{\prime}\not \in \csupp_{\mathbb{Z}}(\Pi)\cup \csupp_{\mathbb{Z}}(\nu^{-1/2}\pi)$. By the induction hypothesis, we are done.

On the other hand, if $\rho_{\sigma(t)} \neq \nu^{1/2}b(\Delta_s)$ we switch the position of $\rho_{\sigma(t)}$ with the factors to its right until it occupies the $k$-th position. We reason exactly as in the previous case but now invoke Lemma \ref{reduction lemma two} instead of Lemma \ref{reduction lemma one} to conclude the argument.

\end{proof}

\section{Proof of Theorem \ref{multiplicity for steinberg} and Theorem \ref{multiplicity for generalized steinberg}} \label{section six}

\subsection{Proof of Theorem \ref{multiplicity for steinberg} Part (A)}

Let notations be as in the statement of Theorem \ref{multiplicity for steinberg}. If $\Pi$ is bad to $\St_{n-1}$ then $\nu^{-(\frac{n-1}{2})}, 
\nu^{-(\frac{n-3}{2})},\ldots,\nu^{(\frac{n-1}{2})}\in \csupp(\Pi)$. Therefore the only possible principal series that is bad to $\St_{n-1}$ is $\Pi=\nu^{-(\frac{n-1}{2})} \times 
\nu^{-(\frac{n-3}{2})} \times \ldots \times \nu^{(\frac{n-1}{2})}$. By Theorem \ref{multiplicity one for generic} we are done.

\subsection{Proof of Theorem \ref{multiplicity for generalized steinberg}}

Let notations be as in the statement of Theorem \ref{multiplicity for generalized steinberg}. If $\Pi$ is bad to $Q(\Delta)$ then, $\nu^{-\frac{1}{2}}\rho,\nu^{\frac{1}{2}} \rho,\ldots, \nu^{m-\frac{1}{2}} \rho, \nu^{m+\frac{1}{2}} \rho\in \csupp(\Pi)$. This implies that $k(m+1)\leq n$. Therefore, \[ km +k \leq km +1\implies k\leq 1. \] This contradicts the fact that $k\geq 2$. We conclude that $\Pi$ is good to $Q(\Delta)$. By Theorem \ref{multiplicity one for generic} we are done.

\subsection{Proof of Theorem \ref{multiplicity for steinberg} Part (B)} \label{subsection six point three}

Given $\Pi = \nu^{-(\frac{n-1}{2})} \times 
\nu^{-(\frac{n-3}{2})} \times \ldots \times \nu^{(\frac{n-1}{2})}\in \Alg(\G_n)$ we want to show that
\[ \dim \Hom_{\G_{n-1}}(\Pi,\St_{n-1}) = n.\] 
The proof is via induction on n. For the base case $n=2$ it is known (see \cite[Theorem 7.4]{vs13}) that,
\[ \dim \Hom_{\G_1}(\nu^{-1/2}\times \nu^{+1/2},1) = 2.\]
We now assume that the theorem is true for all integers less than $n$. We denote $\xi(n) = \nu^{-(\frac{n-1}{2})} \times \ldots \times \nu^{(\frac{n-1}{2})}$. Note that $\xi(n) = \nu^{-(\frac{n-1}{2})} \times \nu^{1/2}\xi(n-1)$. We first prove that the multiplicity is bounded above by $n$, that is, 
\[ \dim \Hom_{\G_{n-1}}(\Pi,\St_{n-1}) \leq n.\]
By Lemma \ref{filtration for parabolically induced modules} we have the short exact sequence,
\[ 0 \rightarrow \nu^{1/2}\xi(n-1) \otimes \zeta^F \rightarrow \xi(n)|_{\G_{n-1}} \rightarrow \nu^{-(\frac{n-2}{2})} \times \nu^{1/2}\xi(n-1)|_{\G_{n-2}} \rightarrow 0 \]
By second adjointness,
\begin{align*}
  \Hom_{\G_{n-1}}&(\nu^{-(\frac{n-2}{2})} \times \nu^{1/2}\xi(n-1)|_{\G_{n-2}},\St_{n-1}) \\ &= \Hom_{\G_1 \times \G_{n-2}}(\nu^{-(\frac{n-2}{2})} \otimes \nu^{1/2}\xi(n-1)|_{\G_{n-2}},\nu^{-(\frac{n-2}{2})} \otimes \nu^{1/2} \St_{n-2}) \\ &= \Hom_{\G_{n-2}}( \nu^{1/2}\xi(n-1)|_{\G_{n-2}},\nu^{1/2} \St_{n-2}).   
\end{align*}
By the induction hypothesis, we conclude that
\begin{equation} \label{eq:induction}
\dim \Hom_{\G_{n-1}}(\nu^{-(\frac{n-2}{2})} \times \nu^{1/2}\xi(n-1)|_{\G_{n-2}},\St_{n-1}) = n-1.
\end{equation} If we now show that $\dim \Hom_{\G_{n-1}}(\nu^{1/2}\xi(n-1) \otimes \zeta^F,\St_{n-1}) = 1$ then we can conclude that,
\[ \dim \Hom_{\G_{n-1}}(\xi(n),\St_{n-1}) \leq n.
\]  
Choose a character $\chi$ such that $\nu^{1/2}\chi$ does not lie in the cuspidal support of $\St_{n-1}$. Applying Lemma \ref{filtration for parabolically induced modules} to $\chi \times \nu^{1/2}\xi(n-1)$, we have the short exact sequence,
\[ 0 \rightarrow \nu^{1/2}\xi(n-1) \otimes \zeta^F \rightarrow (\chi \times \nu^{1/2}\xi(n-1))|_{\G_{n-1}} \rightarrow \chi \times \nu^{1/2}\xi(n-1)|_{\G_{n-2}} \rightarrow 0 \] 
By second adjointness and comparing cuspidal supports we have that \[ \Ext^i_{\G_{n-1}}(\nu^{1/2}\chi \times \nu^{1/2}\xi(n-1)|_{\G_{n-2}}, \St_{n-1})=0\]  for all $i\geq 0$. By a long exact sequence argument we conclude that,
\[ \Hom_{\G_{n-1}}(\nu^{1/2}\xi(n-1) \otimes \zeta^F,\St_{n-1}) = \Hom_{\G_{n-1}}(\chi \times \nu^{1/2}\xi(n-1) ,\St_{n-1}). \]
The $\Hom$ space on the right hand side of the above equation has dimension $1$ by part (A) of Theorem \ref{multiplicity for steinberg}. We obtain that
\begin{equation} \label{an equation for multiplicity} \dim \Hom_{\G_{n-1}}(\nu^{1/2}\xi(n-1) \otimes \zeta^F,\St_{n-1}) = 1, 
\end{equation} which concludes the proof for the upper bound.

We now show that the multiplicity is bounded below by $n$, that is,
\[ \dim \Hom_{\G_{n-1}}(\Pi,\St_{n-1}) \geq n.\]
We use Mackey theory to restrict $\xi(n) = \nu^{-(\frac{n-1}{2})} \times \nu^{1/2}\xi(n-1)$ to $\G_{n-1}$ (see \cite[Section 5]{vs13}). Applying the geometric lemma with respect to the $(1,n-1)$ parabolic subgroup we obtain (see \cite[Section 5, Equation 5.3]{vs13}) the short exact sequence,
\[ 0 \rightarrow \tau \rightarrow \xi(n)|_{\G_{n-1}} \rightarrow \nu^{-(\frac{n-2}{2})} \times \nu^{1/2}\xi(n-1)|_{\G_{n-2}} \oplus \xi(n-1)  \rightarrow 0
 \]
Here, $\tau$ is some subrepresentation of $\xi(n)|_{\G_{n-1}}$ whose definition we do not require.
By Equation \ref{eq:induction} we have that,\[ \dim \Hom_{\G_{n-1}}(\nu^{-(\frac{n-2}{2})} \times \nu^{1/2}\xi(n-1)|_{\G_{n-2}},\St_{n-1}) = n-1.\] Since $\Hom_{\G_{n-1}}(\xi(n-1),\St_{n-1}) = \mathbb{C}$ from the above exact sequence we conclude that,
\[ \dim \Hom_{\G_{n-1}}(\xi(n),\St_{n-1}) \geq (n-1)+1 \geq n. \]
This completes the proof of part (B) of Theorem \ref{multiplicity for steinberg}.

\subsection{Proof of Theorem \ref{intermediate multiplicities}}

Let notations be as in the statement of Theorem \ref{intermediate multiplicities}. We shall use Theorem \ref{multiplicity for steinberg} part (B) to prove Theorem \ref{intermediate multiplicities}. We are given the segment $\Delta_i = [\nu^{-(\frac{n-1}{2})}, \nu^{-(\frac{n-1-2i}{2})}]$, for $i=1, 2, 3, \ldots, (n-2)$. We have the principal series $\Pi_i \in \Alg(\G_n)$ given as,
\[ \Pi_i = Q(\Delta_i) \times \nu^{-(\frac{n-3-2i}{2})} \times \nu^{-(\frac{n-5-2i}{2})} \times \ldots \times \nu^{(\frac{n-1}{2})}. \]
We want to prove that for all $i = 1, 2, 3, \ldots, (n-2)$,
\[ \dim \Hom_{\G_{n-1}}(\Pi_i,\St_{n-1}) = n - i. \] 
We first prove that the desired multiplicity is bounded above by $(n-i)$ and then show that it is bounded below by $(n-i)$ in order to conclude equality. For $0\leq p \leq (n-2)$, we let $\Delta_{1p}$ and $\Delta_{2p}$ denote the segments, 
\[\Delta_{1p} = [\nu^{-(\frac{n-2}{2})}, \nu^{-(\frac{n-2-2p}{2})}] \text{ and }\Delta_{2p} = [\nu^{-(\frac{n-4-2p}{2})},\nu^{(\frac{n-2}{2})}]\] We also denote $\xi(i) = \nu^{-(\frac{i-1}{2})} \times \ldots \times \nu^{(\frac{i-1}{2})}\in Alg(\G_i)$. Note that, $\Pi_i = Q(\Delta_i) \times \nu^{(\frac{i+1}{2})}\xi(n-i-1)$. By Lemma \ref{filtration for parabolically induced modules}, $\Pi_i|_{\G_{n-1}}$ is glued together from, 
\[ \nu^{1/2}Q(\Delta_i) \times \nu^{(\frac{i+1}{2})}\xi(n-i-1)|_{G_{n-i-2}}, \]
\[ \nu^{1/2}Q(^{(1)}\Delta_i) \times (\nu^{(\frac{i+1}{2})}\xi(n-i-1) \otimes \zeta^F), \]
and for $2\leq j\leq (i+1)$,
\[ \nu^{1/2}Q(^{(j)}\Delta_i) \times \RS_{j-2}(\nu^{(\frac{i+1}{2})}\xi(n-i-1)). \]
For $2\leq j < i+1$, by second adjointness,
\begin{align*}
 & \Hom_{\G_{n-1}}(\nu^{1/2}Q(^{(j)}\Delta_i) \times \RS_{j-2}(\nu^{(\frac{i+1}{2})}\xi(n-i-1)),\St_{n-1})  \\ & = \Hom(\nu^{1/2}Q(^{(j)}\Delta_i) \otimes \RS_{j-2}(\nu^{(\frac{i+1}{2})}\xi(n-i-1)),Q(\Delta_{1(i-j)})\otimes Q(\Delta_{2(i-j)})). 
\end{align*}  
Since, $\nu^{-(\frac{n-2}{2})} \not \in \nu^{1/2}.^{(j)}\Delta_i$ by comparing cuspidal supports we conclude that the above $\Hom$ space is equal to 0.

Similarly, we have that \[ \Hom_{\G_{n-1}}(\nu^{1/2}Q(^{(1)}\Delta_i) \times (\nu^{(\frac{i+1}{2})}\xi(n-i-1) \otimes \zeta^F),\St_{n-1}) = 0. \]
For $j=i+1$ we have that,
\[ \dim \Hom_{\G_{n-1}}(\RS_{i-1}(\nu^{(\frac{i+1}{2})}\xi(n-i-1)),\St_{n-1}) = 1. \] The above equality can be shown exactly as in Equation \ref{an equation for multiplicity} in subsection \ref{subsection six point three}. We choose a cuspidal representation of $\G_{i+1}$ instead of a character $\chi$ and the same argument works.
 
Finally, for the first piece in the above filtration,
\begin{align*}
 & \Hom_{\G_{n-1}}(\nu^{1/2}Q(\Delta_i) \times \nu^{(\frac{i+1}{2})}\xi(n-i-1)|_{\G_{n-i-2}},\St_{n-1}) \\ &= \Hom_{\G_{i+1}\times \G_{n-i-2}}(\nu^{1/2}Q(\Delta_i) \otimes \nu^{(\frac{i+1}{2})}\xi(n-i-1)|_{\G_{n-i-2}},Q(\Delta_{1i})\otimes Q(\Delta_{2i})). 
\end{align*}  
Note that $Q(\Delta_{2i}) = \nu^{(\frac{i+1}{2})}\St_{n-i-2}$. By part (B) of Theorem \ref{multiplicity for steinberg},
\[ \dim \Hom_{\G_{n-1}}(\nu^{1/2}Q(\Delta_i) \times \nu^{(\frac{i+1}{2})}\xi(n-i-1)|_{G_{n-i-2}},\St_{n-1}) = n-i-1. \]
We add up the contributions of each of the pieces in the above filtration to the desired $\Hom$ space to conclude that,
\[ \dim \Hom_{\G_{n-1}}(\Pi_i,\St_{n-1}) \leq (n - i - 1) + 1 = (n-i). \]

We now show that the desired multiplicity is bounded below by $(n-i)$. As in the previous subsection, we use Mackey theory to restrict $\Pi_i = Q(\Delta_i) \times \nu^{(\frac{i+1}{2})}\xi(n-i-1)$ to $\G_{n-1}$. Applying the geometric lemma with respect to the $(i+1,n-i-1)$ parabolic subgroup we obtain (see \cite[Section 5, Equation 5.3]{vs13}) the short exact sequence,
\begin{align*}
 0 \rightarrow \tau_i \rightarrow & \Pi_i|_{\G_{n-1}} \rightarrow \nu^{1/2}Q(\Delta_i) \times \nu^{(\frac{i+1}{2})}\xi(n-i-1)|_{\G_{n-i-2}} \\  & \oplus Q(\Delta_i)|_{\G_{i}} \times \nu^{-1/2}\nu^{(\frac{i+1}{2})}\xi(n-i-1)  \rightarrow 0     
\end{align*}
Here, $\tau_i$ is some subrepresentation of $\Pi_i|_{\G_{n-1}}$ whose definition we do not require. We have calculated above that,
\[ \dim \Hom_{\G_{n-1}}(\nu^{1/2}Q(\Delta_i) \times \nu^{(\frac{i+1}{2})}\xi(n-i-1)|_{\G_{n-i-2}},\St_{n-1}) = n-i-1. \]
We now calculate the contribution of the last term of the above exact sequence. By second adjointness,
\begin{align*}
& \Hom_{\G_{n-1}}(Q(\Delta_i)|_{\G_{i}} \times \nu^{-1/2}\nu^{(\frac{i+1}{2})}\xi(n-i-1),\St_{n-1}) \\ &= \Hom_{\G_i \times \G_{n-i-1}}(Q(\Delta_i)|_{\G_{i}} \otimes \nu^{\frac{i}{2}}\xi(n-i-1),Q(\Delta_{1(i-1))})\otimes Q(\Delta_{2(i-1)})). \end{align*}
Note that $Q(\Delta_{2(i-1)}) = \nu^{\frac{i}{2}}\St_{n-i-1}$ and therefore,
\[\Hom_{\G_{n-i-1}}(\nu^{\frac{i}{2}}\xi(n-i-1),\nu^{\frac{i}{2}}\St_{n-i-1}) = \mathbb{C}.\]
Also, it is known (see \cite[Theorem 3]{pr93}) that if $\pi_1\in \Alg(\G_{i+1})$ and $\pi_2\in \Alg(\G_i)$ are irreducible generic representations then $\Hom_{\G_i}(\pi_1,\pi_2) = \mathbb{C}$. We conclude that
\[\Hom_{\G_i}(Q(\Delta_i)|_{\G_{i}},Q(\Delta_{1(i-1))})) = \mathbb{C}\] and therefore,
\[ \dim \Hom_{\G_{n-1}}(Q(\Delta_i)|_{\G_{i}} \times \nu^{-1/2}\nu^{(\frac{i+1}{2})}\xi(n-i-1),\St_{n-1}) = 1. \]
We add up the contributions of the last two terms in the above short exact sequence to conclude that,
\[ \dim \Hom_{\G_{n-1}}(\Pi_i,\St_{n-1}) \geq (n - i - 1) + 1 = (n-i). \]
This completes the proof of the theorem.

\begin{remark}
We have shown that the intermediate multiplicities are attained for specific principal series $\Pi_i$ defined above. Since \[ \Hom_{\G_{n-1}}(\Pi_i,\St_{n-1})\cong \Hom_{\G_{n-1}}(\theta(\Pi_i),\theta(\St_{n-1}))\cong \Hom_{\G_{n-1}}(\theta(\Pi_i),\St_{n-1})\] we obtain another list of principal series $\pi_i^{\prime} = \theta(\Pi_i)$ for which the intermediate multiplicities are attained. We record our observation in the following corollary.
\end{remark}

\begin{corollary}
Let $n\geq 3$ be an integer. For each integer $i = 1, 2, \ldots, (n-2)$, consider the segment $\Delta_i^{\prime} = [\nu^{(\frac{n-1-2i}{2})},\nu^{(\frac{n-1}{2})}]$. Let $\Pi_i^{\prime} \in \Alg(\G_n)$ denote the principal series representation,
\[ \Pi_i^{\prime} = \nu^{-(\frac{n-1}{2})} \times \nu^{-(\frac{n-3}{2})} \times \ldots \times \nu^{(\frac{n-3-2i}{2})} \times Q(\Delta_i^{\prime}). \]
Then for all $i = 1, 2, \ldots, (n-2)$,
\[ \dim \Hom_{\G_{n-1}}(\Pi_i^{\prime},\St_{n-1}) = n - i. \]    
\end{corollary}     

\section{Proof of Theorem \ref{list} and Theorem \ref{list of non generic}} \label{section seven}

We shall require the following two propositions.

\begin{proposition} \label{proposition for generic}
Let $n\geq 2$ be an integer. Given segments $\Delta_1,\Delta_2,\ldots,\Delta_k$ consider the principal series representation $\pi_1 = Z(\Delta_1)\times Z(\Delta_2) \times \ldots \times Z(\Delta_k)$ of $\G_n$. Let $\pi_2$ be an irreducible generic representation of $\G_{n-1}$ such that,
\[ \Hom_{\G_{n-1}}(Z(\Delta_1)\times Z(\Delta_2) \times \ldots \times Z(\Delta_k),\pi_2) \neq 0. \]
Then for all $i=1, 2, \ldots, k$, $l_r(\Delta_i)\leq 2$.
\end{proposition}

\begin{proof}

By the Bernstein-Zelevinsky filtration (see \cite[Lemma 3.2]{cs21}), \[ \Hom_{\G_{n-1}}(\nu^{1/2}\pi_1^{(i_0)} , ^{(i_0-1)}\pi_2) \neq 0 \] for some $i_0 \in \{ 1, 2, \dots, n \}$. By \cite[Corollary 2.6]{ch21}, any simple submodule of $^{(i_0-1)}\pi_2$ is generic. Hence from the above equation we conclude that $\pi_1^{(i_0)}$ has a non-degenerate quotient.

But by the product rule for derivatives and Lemma \ref{derivative of ZDelta} we have that $\pi_1^{(i_0)}$ is glued to together from subquotients of the form, \[ Z(\Delta_1^{\prime})\times Z(\Delta_2^{\prime}) \times \ldots \times \ldots Z(\Delta_k^{\prime}), \] where $\Delta_j^{\prime} = \Delta_j$ or $\Delta_j^{-}$. By the product rule for derivatives and Lemma \ref{derivative of ZDelta}, if a subquotient of a principal series of the above form is non-degenerate then $l_r(\Delta_j^{\prime}) \leq 1$ for all $j=1,2, \ldots, k$. This implies that \[ l_r(\Delta_j) \leq 2 \]  for all $j=1,2, \ldots, k$.

\end{proof}

\begin{proposition} \label{proposition for list}

Let $n,m, k,s\in \mathbb{Z}_{\geq 0}$ such that $m,s\geq 1$ and $k\geq 2$. Let $n=ms+1$. Let $\Delta_1,\Delta_2,\ldots,\Delta_k$ be segments such that the $\Delta_i\cap \Delta_j= \phi$ for all $i,j\in \{ 1,2,\ldots, k\}$ with $i\neq j$ and $\sum_{r=1}^{k}l_a(\Delta_r)=n$.  Consider the principal series representation $\pi_1 = Z(\Delta_1)\times Z(\Delta_2) \times \ldots \times Z(\Delta_k)$ of $\G_n$. Let $\Delta$ be the segment $\Delta = [\rho,\nu^{s-1}\rho]$ where, $\rho$ is a cuspidal representation of $G_m$. Suppose that at least two  segments in the multiset $\mathfrak{m} = \{ \Delta_1,\Delta_2,\ldots,\Delta_k \}$ have relative length greater than 1. Then,
\[ \Hom_{\G_{n-1}}(Z(\Delta_1)\times Z(\Delta_2) \times \ldots \times Z(\Delta_k),Q(\Delta)) = 0. \]

\end{proposition}

\begin{proof}
Since any of the factors $Z(\Delta_i)$ of $\pi_1$ is a quotient of a product of cuspidal representations ($b(\Delta_i) \times \nu^{-1}b(\Delta_i) \times \nu^{-2}b(\Delta_i) \times \ldots \times a(\Delta_i) \twoheadrightarrow Z(\Delta_i)$), we can without loss of generality assume that exactly two segments $\Delta_p$ and $\Delta_q$ in $\{ \Delta_1,\Delta_2,\ldots,\Delta_k \}$ have a relative length greater than one. Set $\pi_2 = Q(\Delta)$. For the sake of contradiction assume that $\Hom_{\G_{n-1}}(\pi_1,\pi_2) \neq 0$. 

By the Bernstein-Zelevinsky filtration, $\Hom_{\G_{n-1}}(\nu^{1/2}\pi_1^{(i_0)} , ^{(i_0-1)}\pi_2) \neq 0$ for some $i_0 \in \{ 1, 2, \dots, n \}$. By Lemma \ref{derivative of ZDelta}, $\pi_1$ is not generic and therefore $i_0 \neq n$. Note that, $^{(i_0-1)}\pi_2 = Q(\Delta^{(i_0-1)})$ is generic and therefore $\pi_1^{(i_0)}$ must also be generic. But by Lemma \ref{derivative of ZDelta}, $Z(\Delta_p)$ and $Z(\Delta_q)$ are not generic. Therefore $\pi_1^{(i_0)}$ must be of the form,
\[ \pi_1^{(i_0)} = Z(\Delta_1)^{\prime}\times \ldots \times Z(\Delta_{p-1})^{\prime} \times Z(\Delta_p^{-}) \times \ldots \times Z(\Delta_{q-1})^{\prime} \times Z(\Delta_q^{-}) \times \ldots Z(\Delta_k)^{\prime}. \]
Here $Z(\Delta_j)^{\prime} = Z(\Delta_j)$ or the trivial representation of $G_0$. Since, 
\[ \Hom_{\G_{n-1}}(\nu^{1/2}\pi_1^{(i_0)} , Q(\Delta^{(i_0-1)})) \neq 0 \] by comparing cuspidal supports we conclude that $\Delta_p=[\nu^a\rho,\nu^b\rho]$ and $\Delta_p=[\nu^c\rho,\nu^d\rho]$ for some real numbers $a,b,c,d$ such that $b-a,d-c\in \mathbb{Z}_{>0}$. Also since $\Delta^{(i_0-1)}$ is a segment, the nonzeroness of the above $\Hom$ space implies that $a-c\in \mathbb{Z}$. Let us assume $a<c$. The case $a>c$ can be handled similarly. Since $a<c$ and $\Delta_p\cap \Delta_q = \phi$ we conclude that $b<c$. 

Since $\Hom_{\G_{n-1}}(\nu^{1/2}\pi_1^{(i_0)} , Q(\Delta^{(i_0-1)})) \neq 0$ and $\nu^b\rho\not \in \Delta_p^{(-)}$ we conclude that $\nu^a\rho, \nu^c\rho\in$ csupp($Q(\Delta^{(i_0-1)}))$) but $\nu^b\rho \not \in$ csupp($Q(\Delta^{(i_0-1)}))$). Since $b-a,c-b \in \mathbb{Z}_{>0}$ this contradicts the fact that $\Delta^{(i_0-1)}$ is a segment. 

\end{proof}

\subsection{Proof of Theorem \ref{list}} 

We are given the principal series $\Pi = \nu^{-(\frac{n-1}{2})} \times 
\nu^{-(\frac{n-3}{2})} \times \ldots \times \nu^{(\frac{n-1}{2})}\in \Alg(\G_n)$. Let $\Delta^0$ denote the segment $\Delta^0 = [\nu^{-(\frac{n-1}{2})},\nu^{(\frac{n-1}{2})}]$. Let $\pi \in \Alg(\G_n)$ be an irreducible subquotient of $\Pi$ such that  $\Hom_{\G_{n-1}}(\pi,\St_{n-1}) = \mathbb{C}$. By an application of the Gelfand-Kazhdan automorphism, we have that $\Hom_{\G_{n-1}}(\pi^{\vee},\St_{n-1}) \neq 0$. 

Suppose that $\pi = Z(\Delta_1,\Delta_2,\ldots,\Delta_k)$ in terms of the Zelevinsky classification. Since $\pi \hookrightarrow Z(\Delta_1)\times Z(\Delta_2) \times \ldots \times Z(\Delta_k)$ therefore by taking duals, $Z(\Delta_1^{\vee})\times Z(\Delta_2^{\vee}) \times \ldots \times Z(\Delta_k^{\vee}) \twoheadrightarrow \pi^{\vee}$. We conclude that,
\[ \Hom_{\G_{n-1}}(Z(\Delta_1^{\vee})\times Z(\Delta_2^{\vee}) \times \ldots \times Z(\Delta_k^{\vee}),\St_{n-1}) \neq 0. \]

Now by Proposition \ref{proposition for generic} and Proposition \ref{proposition for list} at most one of the segments $\Delta_1,\Delta_2,\ldots,\Delta_k$ has length equal to 2. Therefore, the number of possibilities for $\pi$ such that $\Hom_{\G_{n-1}}(\pi ,\St_{n-1}) \neq 0$ is equal to the number of ways of partitioning $\Delta^0$ as a disjoint union of segments with at most one segment having length 2.  We write the disjoint union, 
\[ \Delta^0 = \coprod_{i=1}^k \Delta_k \] in such a way that $\Delta_i$ does not precede $\Delta_j$ for $i<j$ in order to ensure compatibility with the Zelevinsky classification.

Consider the partition $(1, 1, 1, \ldots, 1)$ of $n$ where $1$ occurs '$n$' times. This clearly corresponds to the partition $\Delta^0 = \coprod_{i=1}^n \{\nu^{(\frac{n-i}{2})}\}$. This partition in turn corresponds to the generic representation $\pi_0$ in our list in Theorem \ref{list}.

On the other hand, for $i\in \{ 1, 2, \dots, (n-1)\}$ consider the partition of $n$ where $2$ occurs at only the $i$-th position and $1$ occurs at the remaining positions. This partition clearly corresponds to the representation $\pi_i$ in our list in Theorem \ref{list}.

We have shown that there are at most $n$ irreducible subquotients $\pi_0,\pi_1,\ldots,$ $\pi_{n-1}$ that can have $\St_{n-1}$ as a quotient upon restriction. We now invoke part (B) of Theorem \ref{multiplicity for steinberg}. By the multiplicity one property for irreducible representations each irreducible subquotient of $\Pi$ can make a contribution of at most one to the dimension of the space $\Hom_{\G_{n-1}}(\pi,\St_{n-1})$. Since the dimension of $\Hom_{\G_{n-1}}(\pi,\St_{n-1})$ is equal to $n$ there must be at least $n$ such irreducible subquotients. We conclude that there are exactly $n$ irreducible subquotients of $\Pi$ that have $\St_{n-1}$ as a quotient upon restriction. These are exactly the irreducible representations $\pi_0,\pi_1,\ldots,\pi_{n-1}$ listed in Theorem \ref{list}.

\subsection{Proof of Theorem \ref{list of non generic}}

We reformulate Theorem \ref{list of non generic}
as follows.
\begin{theorem}
 Let $\Delta=[\rho,\nu^{s}\rho]$ be a segment such that $l_a(\Delta)=n-1$ and $\rho$ be some cuspidal representation of $\G_r$. Let $\tau$ be a non-generic irreducible representation of $\G_n$. Let $\pi_1, \pi_2, \ldots ,\pi_{n-1}$ be irreducible representations of $\G_n$ as defined in Theorem \ref{list}.

 \begin{itemize}
     \item[(A)] If $r\geq 2$, then $\Hom_{\G_{n-1}}(\tau,Q(\Delta))=0$.
     \item[(B)] If $r=1$, $\rho=\nu^{-\frac{(n-2)}{2}}$ and $s=n-2$, (that is, $Q(\Delta)=\St_{n-1}$) then
     \[\Hom_{\G_{n-1}}(\tau,Q(\Delta))\neq 0 \iff  \tau = \pi_{i_0}  \text{ for some }  i_0 \in \{ 1,2,\ldots,(n-1) \}.\]
 \end{itemize}
 
\end{theorem}
\begin{proof}
 Suppose that $\Hom_{\G_{n-1}}(\tau,Q(\Delta))\neq 0 $. Let $\tau = Z(\Delta_1,\Delta_2,\ldots,\Delta_k)$ in terms of the Zelevinsky classification. Since, $\tau$ is non-generic, we conclude that at least one of the segments, say $\Delta_p$, has relative length greater than 1. Since $\pi \hookrightarrow Z(\Delta_1)\times Z(\Delta_2) \times \ldots \times Z(\Delta_k)$ therefore by taking duals we obtain that, $Z(\Delta_1^{\vee})\times Z(\Delta_2^{\vee}) \times \ldots \times Z(\Delta_k^{\vee}) \twoheadrightarrow \pi^{\vee}$. We conclude that,
\[ \Hom_{\G_{n-1}}(Z(\Delta_1^{\vee})\times Z(\Delta_2^{\vee}) \times \ldots \times Z(\Delta_k^{\vee}),Q(\Delta^{\vee})) \neq 0. \] By part (A) of Lemma \ref{left right bz filtration} we conclude that there exists an integer $i\geq 1$ such that $\Hom_{G_{n-i}}(\nu^{1/2}.(Z(\Delta_1^{\vee})\times Z(\Delta_2^{\vee}) \times \ldots \times Z(\Delta_k^{\vee}))^{(i)},^{(i-1)}Q(\Delta^{\vee}))\neq 0$. Therefore, 
\begin{equation} \label{eq:one}
 \Hom_{G_{n-i}}(\nu^{1/2}.(Z(\Delta_1^{\vee})^{\prime}\times\ldots  \times Z(\Delta_k^{\vee})^{\prime}),Q((\Delta^{\vee})^{(i-1)}))\neq 0   
\end{equation} 
where, $Z(\Delta_j^{\vee})^{\prime} = Z(\Delta_j^{\vee})$ or $Z(\Delta_j^{\vee})^{\prime} = Z((\Delta_j^{\vee})^{(-)})$. Since, $Q((\Delta^{\vee})^{(i-1)})$ is generic we must have that $Z(\Delta_p^{\vee})^{\prime}$ is generic. Therefore, $\Delta_p$ has relative length equal to 2 and $Z(\Delta_p^{\vee})^{\prime}=Z((\Delta_p^{\vee})^{(-)})$. Suppose $\Delta_p=[\rho_1,\nu\rho_1]$, for some cuspidal representation $\rho_1$. Since    $Z(\Delta_p^{\vee})^{\prime}=Z((\Delta_p^{\vee})^{(-)})=\nu^{-1}\rho_1^{\vee}$, by comparing cuspidal supports we obtain that $\nu^{-1/2}\rho_1^{\vee}\in (\Delta^{\vee})^{(i-1)}$. Suppose $\nu^{-1/2}\rho_1^{\vee}=\nu^{m}\rho^{\vee}$ for some integer $m$. Since $\nu^{-s}\rho^{\vee}, \nu^{-s+1}\rho^{\vee},\ldots, \nu^{m}\rho^{\vee}\in (\Delta^{\vee})^{(i-1)}$ by Equation \ref{eq:one} we conclude
that all the $\nu^{-s-\frac{1}{2}}\rho^{\vee}, \nu^{-s+\frac{1}{2}}\rho^{\vee},\ldots,$ $ \nu^{m-\frac{1}{2}}\rho^{\vee}$ belong to
$\cup_{q=1}^k \Delta_q^{\vee}$. By taking duals we conclude that all the $\nu^{-m+\frac{1}{2}}\rho,$ $ \nu^{-m+\frac{3}{2}}\rho,\ldots, \nu^{s+\frac{1}{2}}\rho$ belong to
$\cup_{q=1}^k \Delta_q$.

On the other hand, by Lemma \ref{left right bz filtration} part (B) we have that,  
\begin{equation} \label{eq:two}
 \Hom_{G_{n-j}}(\nu^{-1/2}.^{(j)}(Z(\Delta_1^{\vee})\times \ldots \times Z(\Delta_k^{\vee})),Q( ^{(j-1)}(\Delta^{\vee})))\neq 0,
\end{equation}  
for some integer $j\geq 1$. Just as above, by comparing cuspidal supports we obtain that $\nu^{-1/2}Z( ^{(-)}(\Delta_p^{\vee}))=\nu^{-1/2}\rho_1^{\vee}=\nu^{m}\rho^{\vee} \in $ $ ^{(j-1)}(\Delta^{\vee})$. Using the fact that
$\nu^{m}\rho^{\vee}, \nu^{m+1}\rho^{\vee},\ldots, \rho^{\vee} \in $ $^{(j-1)}(\Delta^{\vee})$ by Equation \ref{eq:two} we conclude that $\nu^{m+\frac{1}{2}}\rho^{\vee}, \nu^{m+\frac{3}{2}}\rho^{\vee},\ldots, \nu^{\frac{1}{2}}\rho^{\vee} \in \cup_{q=1}^k \Delta_q^{\vee}$. By taking duals we conclude that $\nu^{-\frac{1}{2}}\rho,\ldots, \nu^{-m-\frac{3}{2}}\rho, \nu^{-m-\frac{1}{2}}\rho \in \cup_{q=1}^k \Delta_q$.

From the above arguments we obtain that, $\nu^{-\frac{1}{2}}\rho,\nu^{\frac{1}{2}}\rho,\ldots,\nu^{s+\frac{1}{2}}\rho\in \cup_{q=1}^k \Delta_q$. This implies that,\[ n=\sum_{q=1}^k l_a(\Delta_q)\geq \sum_{u=0}^{s+1} n(\nu^{-\frac{1}{2}+u}\rho) = (s+2)r = (s+1)r + r = (n-1) +r. \]

If $r\geq 2$, we obtain a contradiction. This proves part (A).

We now prove part (B). If $r=1$, then the above inequality becomes an equality. When $\rho=\nu^{-\frac{(n-2)}{2}}$ and $s=n-2$ this implies that $\csupp(\tau)= \{ \nu^{-(\frac{n-1}{2})}, 
\nu^{-(\frac{n-3}{2})}, \ldots,\nu^{(\frac{n-1}{2})} \}$. This means that $\tau$ is a subquotient of the principal series $\nu^{-(\frac{n-1}{2})} \times 
\nu^{-(\frac{n-3}{2})} \times \ldots \times \nu^{(\frac{n-1}{2})}$. By Theorem \ref{list} we are done.

\end{proof}

\section{Proof of Theorem \ref{multiplicity for ZDelta}} \label{section eight}

Given $k\geq 2$ and $n = km +1$. Let $\Delta$ be the segment $\Delta = [\rho,\nu^{m-1}\rho]$ where, $\rho$ is a cuspidal representation of $\G_k$ and $m\geq 1$ is an integer. Any irreducible representation of $\G_n$ can be obtained as a quotient of some product of cuspidal representations. Therefore it is enough to prove the upper bound assuming that $\Pi$ is a principal series representation of $\G_n$ parabolically induced from a cuspidal representation of a Levi subgroup. 

If $m=1$ then $Z(\Delta)=\rho$ is simply a cuspidal representation of $\G_k$ and we invoke Theorem \ref{multiplicity for generalized steinberg}. So from now on we assume that $m\geq 2$.

\textbf{Case 1} - Suppose that $\Pi$ is of the form $\Pi = \nu^{1/2}(\nu^{m-1}\rho \times \nu^{m-2}\rho \times \ldots \times \rho)\times \chi$ for some character $\chi$. By part (A) of Lemma \ref{left right bz filtration} and Lemma \ref{derivative of ZDelta},
\begin{align*}  
\dim \Hom_{\G_{n-1}}(\Pi,Z(\Delta)) & \leq \dim \Hom_{\G_{n-1}}(\nu^{1/2}\Pi^{(1)},Z(\Delta)) \\ & + \dim \Hom_{\G_{n-1}}(\nu^{1/2}\Pi^{(k+1)},Z(^{-}\Delta)).
\end{align*}
The first term on the right hand side is clearly $0$. We use the product rule for derivatives to obtain the composition factors of $\nu^{1/2}\Pi^{(k+1)}$. The only composition factor of $\nu^{1/2}\Pi^{(k+1)}$
that can contribute to $\Hom_{\G_{n-1}}(\nu^{1/2}\Pi^{(k+1)},Z(^{-}\Delta))$ is $\nu^{m-1}\rho \times \nu^{m-2}\rho \times \ldots\times \nu \rho$. Therefore,
\[ \dim \Hom_{\G_{n-1}}(\nu^{1/2}\Pi^{(k+1)},Z(^{-}\Delta)) \leq 1 \implies \dim \Hom_{\G_{n-1}}(\Pi,Z(\Delta)) \leq 1. \]

\textbf{Case 2} - Suppose that $\Pi$ is not of the form $\nu^{1/2}(\nu^{m-1}\rho \times \nu^{m-2}\rho \times \ldots \times \rho)\times \chi$. Clearly $\Hom_{\G_{n-1}}(\nu^{-1/2}.^{(1)}\Pi,Z(\Delta)) = 0$ and therefore by part (B) of Lemma \ref{left right bz filtration} and Lemma \ref{derivative of ZDelta} we conclude that, 
\[ \dim \Hom_{\G_{n-1}}(\Pi,Z(\Delta)) \leq \dim \Hom_{\G_{n-1}}(\nu^{-1/2}.^{(k+1)}\Pi,Z(\Delta^{-})). \]
If $\Hom_{\G_{n-1}}(\Pi,Z(\Delta)) \neq 0$ then $\Hom_{\G_{n-1}}(\nu^{-1/2}.^{(k+1)}\Pi,Z(\Delta^{-})) \neq 0$. Therefore each of $\{ \rho, \nu \rho, \nu^2 \rho,\ldots, \nu^{m-2} \rho \}$ lies in the cuspidal support of $\nu^{-1/2}\Pi$. If $\dim \Hom_{\G_{n-1}}(\Pi,Z(\Delta)) \geq 2$, then $\dim \Hom_{\G_{n-1}}(\nu^{-1/2}.^{(k+1)}\Pi,Z(\Delta^{-})) \geq 2$. This means that the composition factor $\nu^{m-2}\rho\times \nu^{m-3} \rho\times \ldots\times \nu \rho \times  \rho$ must occur at least twice in the semi-simplification of $\nu^{-1/2}.^{(k+1)}\Pi$. This can only happen for the principal series $\Pi_1,\Pi_2,\ldots,\Pi_{m-1}$ defined as,
\[ \Pi_i = \tau_{i1} \times \tau_{i2} \times \ldots \times \tau_{im} \times \chi_i \] where,
\[ \tau_{ij} = \begin{cases}
     \nu^{1/2}.\nu^{m-j-1}\rho, & \text{ for } 1\leq j < i \\
     \nu^{1/2}.\nu^{m-i-1}\rho, & \text{ for }  j=i, i+1 \\
     \nu^{1/2}.\nu^{m-j-1}\rho, & \text{ for } i+2 \leq j \leq m
\end{cases}
\]
and $\chi_i$ ($i=1, 2, \ldots, (m-1)$) are some characters.

For instance, for $\Pi_1 = \nu^{1/2}(\nu^{m-2}\rho\times \nu^{m-2}\rho\times \nu^{m-3} \rho\times\ldots\times \nu^2 \rho \times \rho)\times \chi_1$, by the product rule for derivatives it is clear that $\nu^{m-2}\rho\times \nu^{m-3} \rho\times \ldots\times \nu \rho \times  \rho$ occurs exactly twice in the semi-simplification of $\nu^{-1/2}.^{(k+1)}\Pi_1$. 

Therefore, we need to now only deal with the principal series $\Pi_1,\Pi_2,\ldots,$ $\Pi_{m-1}$ defined above. Since $Z(\Delta)\hookrightarrow \rho \times \nu^2 \rho \times \ldots \times \nu^{m-1}\rho$,
\[ \dim \Hom_{\G_{n-1}}(\Pi_i,Z(\Delta)) \leq \dim \Hom_{\G_{n-1}}(\Pi_i,\rho \times \nu^2 \rho \times \ldots \times \nu^{m-1}\rho)\] for $i=1, 2, \ldots, (m-1)$. Note that each $\Pi_i$ is a standard representation and $\rho \times \nu^2 \rho \times \ldots \times \nu^{m-1}\rho$ is the dual of a standard representation. Therefore by the multiplicity one theorem for standard representations (see \cite[Theorem 1.1]{ch23}) the dimension of the $\Hom$ space on the right hand side of the above inequality is equal to $1$.

\section{Geometric Origin of Hom Spaces: An Example} \label{section nine}

In section 9 of the work \cite{pr18}, it is suggested by Dipendra Prasad that homomorphisms and extensions between representations realized on functions on geometric spaces or $\ell$-sheaves often have a geometric origin via the Bernstein-Zelevinsky exact sequence, cf. Proposition 1.8 of \cite{bz76}. Let $\pi_1,\ldots,\pi_{n-1}$ be the irreducible representations of $\GL_n(F)$ as defined in Theorem \ref{list}. We have seen that $\Hom_{\G_{n-1}}(\pi_i,\St_{n-1})= \mathbb{C}$ for all $i= 1, 2, \ldots, (n-1)$. In this section, we explain how these homomorphisms arise from maps between geometric spaces. 

\subsection{Geometric Origin of Certain Hom Spaces} \label{section nine point one} Given a $p$-adic manifold $X$, let $\mathbb{S}[X]$ denote the space of locally constant compactly supported functions on X. For our purposes, a $p$-adic manifold is an open subset of a topological space of the form $G/H$ where, $G$ is a $p$-adic group and $H$ is a closed subgroup of $G$. Let $V$ be a vector space over the non-archimedean field $F$ with basis $\{ e_1,e_2,\ldots, e_n$\}. Let $W$ be the subspace of $V$ generated by the first $(n-1)$ elements in this basis. Then $\GL(W)=\GL(n-1)$ is contained inside $\GL(V)=\GL(n)$, acting trivially on $e_n$. 

For $r=1, 2, \ldots, (n-1)$, let $Q_r$ be the parabolic subgroup in $\GL(V)$ defined as the stabilizer of the flag, \[ V_{1,r} \subset V_{2,r} \subset \ldots \subset V_{r,r} \subset \ldots \subset V_{n-1,r} \]
where $V_{i,r} = 
    \begin{cases}
    \langle e_1,\ldots, e_{i} \rangle, & \text{ for } i=1,2,\ldots,r-1, \\
    \langle e_1,\ldots, e_{r},e_n \rangle, & \text{ for } i=r, \\
    \langle e_1,\ldots, e_r, e_n, e_{r+1}, ,\ldots, e_{i} \rangle, & \text{ for } i=(r+1),\ldots, (n-1).
    \end{cases}$

Clearly, $Q_r$ is a parabolic subgroup in $\GL(V)$ corresponding to the partition of $n$ given by $(1,1,\ldots,2,\ldots,1)$  where $2$ occurs at the $r$-th position. Let $B_n$ and $B_{n-1}$ denote the standard Borel subgroups in $\GL(n)$ and $\GL(n-1)$ respectively.

It is clear that $Q_r \cap \GL(n-1) = B_{n-1}$ and therefore we have the injection, \[ \GL(n-1)/B_{n-1} \hookrightarrow \GL(n)/Q_r. \] We also have the following natural injection, \[ \GL(n-1)/B_{n-1} \hookrightarrow \GL(n)/B_n. \] We now prove the following lemma.

\begin{lemma} \label{lemma on parabolic}
  Let $B_{n-1}$ be a fixed Borel subgroup of $\GL(n-1)=\GL(W)$, and let $P$ be a parabolic subgroup of $\GL(n)=\GL(V)$ where, $V=W\oplus e_n$. Then the natural inclusion of $\GL(n-1)$ inside $\GL(n)$ gives rise to a closed embedding, \[ \GL(n-1)/B_{n-1} \hookrightarrow \GL(n)/P, \]
  i.e., $P\cap \GL(n-1) = B_{n-1}$,  if and only if, in the notation introduced above, either $P=B_n$ or $P=Q_r$ for some $r\in \{ 1, 2, \ldots, (n-1) \}.$ In particular, if $Q^{\prime}$ is a parabolic subgroup in $\GL(n)$ strictly containing $Q_r$ for some $r$, then $Q^{\prime}\cap \GL(n-1)$ is a parabolic subgroup of $\GL(n-1)$ strictly containing $B_{n-1}$.
\end{lemma}
\begin{proof}
The proof depends on realizing any parabolic subgroup $P$ in $\GL(n) = \GL(V)$ as the stabilizer of a flag, \[ 0 \neq V_{i_1} \subsetneq V_{i_2} \subsetneq \ldots \subset V_{i_s} \subsetneq V. \] Therefore if $\GL(n-1)\cap P = B_{n-1}$, then $\{V_{i_j}\cap W \}$ is the standard flag, \[ \langle e_1 \rangle \subset \langle e_1, e_{2} \rangle \subset \ldots \subset \langle e_1,e_2\ldots, e_{n-1} \rangle \] in $W$. Then it can be seen that the following two possibilities can occur. \begin{enumerate}
    \item We have that $s=n-1$ and the dimension of each vector space in the above flag successively increases by $1$. In this case, $P$ is the unique Borel subgroup of $\GL(n)$ containing $B_{n-1}$, which is $B_n$.
    \item We have that $s=n-2$ and the dimension of each vector space in the above flag successively increases by $1$ except exactly at one point where it increases by $2$. In this case, one can check that $P$ is one of the parabolic subgroups $Q_r$ defined above.
\end{enumerate} Note that $\GL(n-1)/B_{n-1}$ being a projective variety, is compact, and hence its image under the above continuous, injective map inside $\GL(n)/Q_r$ and $\GL(n)/B_n$ is closed. 
\end{proof}

The following lemma generalizes the previous one for all classical groups whose proof we omit. 
\begin{lemma} \label{general lemma on parabolic}
Let $(G,H)$ be any of the following pairs: $(\GL(n),\GL(n-1))$, $(\SO(n), \SO(n-1))$ or $(\U(n), \U(n-1))$ with both $G$ and $H$ quasi-split in the latter two cases, and assume furthermore that the rank of $G$ is greater than the rank of $H$. Let $B_H$ be a fixed Borel subgroups of $H$ and $Q$ a parabolic subgroup of $G$ such that $Q\cap H \supset B_H$.
Then $Q$ has a conjugate $Q^{\prime}$ such that $Q^{\prime}\cap H$ is equal to $B_H$ if and only if the semi-simple rank of any Levi subgroup of $Q$ is less than or equal to 1. Therefore if there exists a closed embedding, \[ f: H/B_H \hookrightarrow G/Q, \] then any Levi subgroup of $Q$ has semi-simple rank less than or equal to 1. Conversely, if $Q$ has semi-simple rank less than or equal to 1 then it has a conjugate $Q^{\prime}$ such that there exists a closed embedding, \[ f: H/B_H \hookrightarrow G/Q^{\prime}. \] 
\end{lemma}

Let us now see how using Lemma \ref{lemma on parabolic} a homomorphism from $\pi_r$ to $\St_{n-1}$ can arise geometrically. For $1\leq r\leq (n-1)$, the closed embedding \[ \GL(n-1)/B_{n-1} \hookrightarrow \GL(n)/Q_r\] gives rise to the surjective map (see \cite[Proposition 1.8]{bz76}), \[ \mathbb{S}[\GL(n)/Q_r] \twoheadrightarrow \mathbb{S}[\GL(n-1)/B_{n-1}]. \] 

We have the natural map $\mathbb{S}[\GL(n-1)/B_{n-1}] \twoheadrightarrow \St_{n-1}$ which gives us the surjective map, \[ \mathbb{S}[\GL(n)/Q_r] \twoheadrightarrow \St_{n-1}. \]

Let $P_r$ denote the standard parabolic subgroup inside $\GL_n(F)$ corresponding to the partition of $n$ given by $(1,1,\ldots,2,\ldots,1)$ where $2$ occurs at the $r$-th position. Since $P_r$ and $Q_r$ are conjugate to each other it is clear that $\mathbb{S}[\GL(n)/Q_r]\cong \mathbb{S}[\GL(n)/P_r]$ considered as representations of $\GL(n)$ with the usual right regular action. Notice that $\mathbb{S}[\GL(n)/P_r]$ is the same as the representation,  \[ \nu^{-(\frac{n-1}{2})} \times \ldots \times  \nu^{(\frac{-n+2r-3}{2})} \times Z([ \nu^{(\frac{-n+2r-1}{2})}, \nu^{(\frac{-n+2r+1}{2})}]) \times \nu^{(\frac{-n+2r+3}{2})} \times \ldots \times \nu^{(\frac{n-1}{2})}. \]

We claim that $\pi_{n-r}$ is a quotient of $\mathbb{S}[\GL(n)/P_r]$. This is because, borrowing notations from Theorem \ref{list}, the above principal series is equal to,
\[ Z(\Delta_{(n-r)(n-1)}) \times Z(\Delta_{(n-r)(n-2)}) \times \ldots Z(\Delta_{(n-r)(n-r)}) \times \ldots Z(\Delta_{(n-r)1}). \]
We recall that in the above expression $\Delta_{(n-r)(n-r)}$ is the only segment having length 2 and that, 
\[\Delta_{(n-r)(n-r)} = [ \nu^{(\frac{-n+2r-1}{2})}, \nu^{(\frac{-n+2r+1}{2})}] .\] This shows that $\pi_{n-r}$ is a quotient of $\mathbb{S}[\GL(n)/P_r]$.

Next we prove that if $i\neq (n-r)$ then $\pi_i$ is not a subquotient of $\mathbb{S}[\GL(n)/P_r]$. This is because if $\pi_i$ were a subquotient of $\mathbb{S}[\GL(n)/P_r]$ then any irreducible subquotient of $\pi_i^{(n-1)}$ would be a subquotient of $\mathbb{S}[\GL(n)/P_r]^{(n-1)}$. Observe that $\pi_i^{(n-1)}=\nu^{(\frac{n-2i-1}{2})}$ (see \cite[Theorem 8.1]{ze80}) whereas $\mathbb{S}[\GL(n)/P_r]^{(n-1)}= \nu^{(\frac{-n+2r-1}{2})}$. Since $\nu^{(\frac{n-2i-1}{2})}\neq \nu^{(\frac{-n+2r-1}{2})}$ if $i\neq (n-r)$ we conclude that for $i\neq (n-r)$, $\pi_i$ is not a subquotient of $\mathbb{S}[\GL(n)/P_r]$.

We now claim that no irreducible subquotient of $\mathbb{S}[\GL(n)/P_r]$ other than $\pi_{n-r}$ has a non-trivial map to $\St_{n-1}$. Let $\tau$ be an irreducible subquotient of $\mathbb{S}[\GL(n)/P_r]$ other than $\pi_{n-r}$. Suppose that $\tau = Z(\Delta_1,\Delta_2,\ldots,\Delta_k)$ in terms of the Zelevinsky classification. Since, $\tau\neq \pi_i$ for $i=1,2,\ldots,(n-1)$ and $\tau$ is non generic, we conclude that either some segment $\Delta_i$ has length greater than 2 or that at least two of the segments have length greater than 1.  Since, \[ \tau \hookrightarrow Z(\Delta_1)\times Z(\Delta_2) \times \ldots \times Z(\Delta_k) \] by taking duals we conclude that $\tau$ is a quotient of the principal series $Z(\Delta_1^{\vee})\times Z(\Delta_2^{\vee}) \times \ldots \times Z(\Delta_k^{\vee})$. Therefore by Proposition \ref{proposition for generic} and Proposition \ref{proposition for list}, we conclude that irreducible subquotients of the principal series $\mathbb{S}[\GL(n)/P_r]$ other than $\pi_{n-r}$ do not have a non-trivial map to $\St_{n-1}$. Therefore given the map, 
 \[ \mathbb{S}[\GL(n)/P_r]\cong \mathbb{S}[\GL(n)/Q_r] \twoheadrightarrow \St_{n-1} \] going modulo the irreducible subquotients other than $\pi_{n-r}$ we get the map,
\[ \pi_{n-r} \twoheadrightarrow \St_{n-1}. \]
This gives us a geometric realization of a homomorphism from $\pi_{n-r}$ to $\St_{n-1}$ for all $r=1,2,\ldots,(n-1)$.

In the next section, we will give another purely geometric argument to say that any $G$-equivariant map from 
$\mathbb{S}(G/P_r)$ to $\St_{n-1}$ factors through $\pi_{n-r}$. We have preferred to give the above argument using the Zelevinsky 
classification as Propositions \ref{proposition for generic} and \ref{proposition for list} are essential for us to prove that non-trivial homomorphisms from other irreducible subquotients of
$\mathbb{S}(G/P_r)$ to $\St_{n-1}$ do not exist in most cases, which cannot be deduced from geometric arguments which are
good to construct homomorphisms but not to say none exist!

We also point out that Lemma \ref{lemma on parabolic} gives us a geometric reason for why $\pi_1,\pi_2,\ldots$, $\pi_{n-1}$ have $\St_{n-1}$ as a quotient upon restriction. Any irreducible non-generic representation of $\GL_n(F)$ can be realized as a quotient of a principal series representation arising from a parabolic subgroup $P$ corresponding to a partition $(n_1,n_2,\ldots,n_k)$ of $n$ such that $n_i>1$ for some $i\in \{1,2,\ldots,k\}$. If we believed that homomorphisms between such representations should arise from morphisms between geometric spaces we expect that there should exist an injection, \[ \GL(n-1)/B_{n-1} \hookrightarrow \GL(n)/P.\] By Lemma \ref{lemma on parabolic} this is possible only if exactly one of the $n_i$ is equal to 2 and all the other $n_j$'s are equal to 1. These parabolic subgroups correspond to the irreducible non-generic representations $\pi_1,\pi_2,\ldots,\pi_{n-1}$. This gives a heuristic argument why other irreducible non-generic representations do not have $\St_{n-1}$ as a quotient upon restriction. However, we point out that such heuristic using geometric arguments does not work in all situations (see Remark \ref{remark in section 9 on heuristic}).

\subsection{A Question for Classical Groups}

The statement of Lemma \ref{lemma on parabolic} can be reformulated as follows. There exists an injection $\GL(n-1)/B_{n-1} \hookrightarrow \GL(n)/P$ if and only if the rank of $[L,L]$ is less than or equal to $1$, where $L$ is a Levi subgroup of the parabolic subgroup $P$. Here, $[L,L]$ denotes the derived subgroup of $L$.

In light of this reformulation, the geometric considerations considered above for the general linear group can also be considered for other classical groups, say $\SO(n-1) \subset \SO(n)$ with $\SO(n-1)$ quasi-split. In this case, we are led to the following question which we do not know how to answer. The question below can analogously be formulated for more general principal series as well, but we content ourselves with this simplest case.
 
\begin{question} \label{question in section 9}
Let $W$ be a codimension 1 nondegenerate subspace of a quadratic space $V$ of dimension $n$ over a non-archimedean local field $F$. Consider the pair $(G,H) = (\SO(n), \SO(n-1))= (\SO(V), \SO(W))$ and assume that the rank of $G$ is greater than rank of $H$. Assume $H$ is quasi-split, and let $P$ be a parabolic subgroup of $G$. Then can we say that the Steinberg representation of $H$ appears as a quotient of the representation of $G$ realized on locally constant functions on $G/P$ if and only if 
the rank of $[L,L]$ is less than or equal to $1$, where $L$ is a Levi subgroup of $P$?
\end{question}

\begin{remark} \label{remark in section 9 on heuristic}
It is only for the Steinberg representation on the smaller group that we can hope to realize $\Hom_{H}(\pi_1,\pi_2)$, $\pi_2$ being the Steinberg of $H$ through geometric means. For example, suppose $\chi_1,\chi_2$ and $\chi_3$ are unitary characters of $\GL_1(F)$. Then by the non-generic Gan-Gross-Prasad conjecture (see \cite{ch22}), the non-generic (irreducible unitary) representation $Z([\chi_1\nu^{-1/2},\chi_1\nu^{1/2}])\times Z([\chi_2\nu^{-1/2},\chi_2\nu^{1/2}])$ of $\GL_4(F)$ has the representation $\chi_1\times \chi_2 \times \chi_3$ of $\GL_3(F)$ as a quotient though it cannot be realized through a map from $\GL_3/B_3$ to $\GL_4/P_{2,2}$ (as none exist!). As mentioned earlier, such geometric arguments are good to construct non-trivial homomorphisms but not to say that none exist. Thus in the next section, for question \ref{question in section 9}, we construct a $G$-equivariant map from locally constant compactly supported functions from $G/P$ to $\St_H$ if the rank of $[L,L]\leq 1$, but these geometric methods do not allow us to prove that none exist if the rank of $[L,L]\geq 2$.
\end{remark}

\section{Alternative Proof of the Lower Bound in Theorem \ref{multiplicity for steinberg} (B)} \label{section ten}

In Subsection \ref{subsection six point three}, we gave an inductive proof for the fact that \[ \dim \Hom_{\G_{n-1}}(\nu^{-(\frac{n-1}{2})} \times 
\nu^{-(\frac{n-3}{2})} \times \ldots \times \nu^{(\frac{n-1}{2})},\St_{n-1}) \geq n. \] In this section, we give an alternative proof of this result from the geometric perspective of the last section. The advantage of this perspective is that it can also be generalized to the context of other classical groups, for instance the pair $(G,H) = (\SO(n), \SO(n-1))$ or $(G,H) = (\U(n), \U(n-1))$ (with both $G$ and $H$ quasi-split). Assuming that the Steinberg representation of $G$ is a projective module under restriction to $H$, we obtain a higher multiplicity result for the pair $(G,H)$. Of course in the context of the general linear group it is known (see \cite{cs21}) that $\St_n|_{\G_{n-1}}$ is a projective module in the category $\Alg(\G_{n-1})$. 

We now go back to our situation of $(G,H)= (\GL(n), \GL(n-1))$. 
Fix a Borel subgroup $B_H$ of $H$, contained in a unique Borel subgroup $B_G$ of $G$. Let $S_1$ (resp. $S_2$) denote the set of all standard parabolic subgroups $P$ of $G$ with a  Levi
subgroup of semi-simple rank equal to 1 (resp. equal to 2). Let $P_1,P_2,\ldots,P_{n-1}$ denote the set of standard parabolic subgroups in $S_1$ (there are $n-1$ of them for the pair $(G,H)= (\GL(n), \GL(n-1))$). Let $P_{ij}$ denote a standard parabolic subgroups belonging to $S_2$ containing $P_i$, thus $P_{ij} \supset P_i \supset B_G$.

We wish to prove that, \[ \dim \Hom_{H}(\mathbb{S}(G/B_G),\St_H) \geq n. \] Observe that for $P_i\in S_1$, $\mathbb{S}(G/P_i)\subset \mathbb{S}(G/B_G)$. Also by definition we have that $\frac{\mathbb{S}(G/B_G)}{\sum_{P_i\in S_1}\mathbb{S}(G/P_i)} \cong \St_G$. Hence we obtain the following short exact sequence of $G = \GL_n(F)$-modules, 
\begin{equation} \label{eq:five}
0 \rightarrow \frac{\sum_{P_i\in S_1}\mathbb{S}(G/P_i)}{\sum_{P_{ij}\in S_2}\mathbb{S}(G/P_{ij})} \rightarrow \frac{\mathbb{S}(G/B_G)}{\sum_{P_{ij}\in S_2}\mathbb{S}(G/P_{ij})} \rightarrow \St_G \rightarrow 0.    
\end{equation} 

For each $P_i\in S_1$, let $K_i = \frac{\mathbb{S}(G/P_i)}{\sum_{P_{ij}\in S_2} \mathbb{S}(G/P_{ij})}$,
where the sum in the denominator is understood to be taken over all those $P_{ij}$ that (strictly) contain $P_i$. We also set, $K = \frac{\sum_{P_i\in S_1}\mathbb{S}(G/P_i)}{\sum_{P_{ij}\in S_2}\mathbb{S}(G/P_{ij})}$. We note that the spaces in the numerator and denominator of $K_i$'s and $K$ are contained in $\mathbb{S}(G/B_G)$, and the sums refer to their sum as subspaces of $\mathbb{S}(G/B_G)$. We claim that the natural map, \[ \psi: \sum_i K_i \rightarrow K \] is an isomorphism. The surjectivity of $\psi$ is clear. On the other hand, by \cite[Theorem 1.1]{ca81} the $K_i$'s are distinct irreducible submodules of $K$. We conclude that $\psi$ is injective and $\sum_i K_i = \oplus_i K_i$. 

Assume that $\St_G|_H$ is projective in $\Alg(H)$. This is certainly true for the pair $(\GL_n,\GL_{n-1})$. In this case the short exact sequence in Equation \ref{eq:five} splits to give a decomposition (treating all representations of $G$ as representations of $H$),
\[ \frac{\mathbb{S}(G/B_G)}{\sum_{P_{ij}\in S_2}\mathbb{S}(G/P_{ij})} = \oplus_{i=1}^{n-1} K_i \oplus \St_G. \] We obtain that, 
\[ \dim \Hom_H(\mathbb{S}(G/B_G),\St_H) \geq \sum_i \dim \Hom_H(K_i,\St_H) + \dim \Hom_H(\St_G,\St_H). \] Since, $\dim \Hom_H(\St_G,\St_H)=1$ if we show that $\dim \Hom_H(K_i,\St_H)\geq 1$ for all $i=1,2,\ldots,(n-1)$, we will have proved that \[ \dim \Hom_{H}(\mathbb{S}(G/B_G),\St_H) \geq n.\] Now let $Q_i$ be the parabolics inside $\GL_{n}(F)$ as defined in Section \ref{section nine point one} with the property that $Q_i\cap H = B_H$, giving rise to the natural maps,
\[ \Tilde{f_i}: \mathbb{S}(G/Q_i) \rightarrow \mathbb{S}(H/B_H) \rightarrow \St_H. \] 

Let $Q_{ij}$ denote a parabolic of $G$ strictly containing $Q_i$. By Lemma \ref{lemma on parabolic}, the image of the submodules $\mathbb{S}(G/Q_{ij})\subset \mathbb{S}(G/Q_i)$ under the map $\Tilde{f_i}$ lies in $\mathbb{S}(H/Q_{H_{ij}})$ where $Q_{H_{ij}}$ is a parabolic subgroup of $H$ containing $B_H$ but
not equal to $B_H$. Since the composition of the maps $\mathbb{S}(H/Q_{H_{ij}})\rightarrow \mathbb{S}(H/B_H) \rightarrow \St_H$ is zero for $Q_{H_{ij}}$ \textit{strictly} containing $B_H$, $\mathbb{S}(G/Q_{ij})$ lies in the kernel of the map, 
\[ \Tilde{f_i}: \mathbb{S}(G/Q_i) \twoheadrightarrow \St_H. \] This gives rise to the natural induced map, 
\[ \bar{f_i}: \frac{\mathbb{S}(G/Q_i)}{\sum_{Q_{ij}\supset Q_i} \mathbb{S}(G/Q_{ij})} \twoheadrightarrow \St_H. \] Since $P_i$ and $Q_i$ are conjugates, we have an isomorphism of $G = \GL_n(F)$-modules:
\[ K_i= \frac{\mathbb{S}(G/Q_i)}{\sum_{Q_{ij}\supset Q_i} \mathbb{S}(G/Q_{ij})}  \cong \frac{\mathbb{S}(G/P_i)}{\sum_{P_{ij}\supset P_i} \mathbb{S}(G/P_{ij})}.  \]
This proves that $\dim \Hom_H(K_i,\St_H)\geq 1$ for all $i=1,2,\ldots,(n-1)$. This combined with the fact that $\dim \Hom_H(\St_G,\St_H)=1$ gives us that, \[ \dim \Hom_{H}(\mathbb{S}(G/B_G),\St_H) \geq n. \]

\end{document}